\documentclass[english,12pt,twoside]{article}
\usepackage{amssymb,amsbsy,amsmath,amsfonts,amssymb,amscd}
\usepackage{latexsym}
\usepackage{euscript}
\usepackage{exscale}
\usepackage{epsfig}
\usepackage{amsthm}
\usepackage[english]{babel}

 \newtheorem{thm}{Theorem}[section]
 \newtheorem{cor}[thm]{Corollary}
 \newtheorem{lem}[thm]{Lemma}
 \newtheorem{prop}[thm]{Proposition}
 \theoremstyle{definition}
 
 \theoremstyle{remark}
 \newtheorem{rem}[thm]{Remark}
 \newtheorem*{ex}{Example}
 \numberwithin{equation}{section}

\title{Compatibility of a Jacobi structure and a Riemannian structure on a Lie algebroid}

\author{Yacine A\"it Amrane, Ahmed Zeglaoui}

\begin{document}

\maketitle

\begin{abstract}
\noindent In a preceding paper we introduced a notion of compatibility between a Jacobi structure and 
a Riemannian structure on a smooth manifold.
We proved that in the case of fundamental examples of Jacobi structures
: Poisson structures, contact structures and locally conformally
symplectic structures, we get respectively Riemann-Poisson
structures in the sense of M. Boucetta, $(1/2)$-Kenmotsu structures
and locally conformally K\"ahler structures. In this paper we are
generalizing this work to the framework of Lie algebroids.  \bigskip  
\end{abstract}
\bigskip\bigskip

\section{Introduction}

In \cite{boucetta1,boucetta2}, M. Boucetta introduced a notion of compatibility
between a Poisson structure and a Riemannian structure on a smooth manifold, he called  
such a pair of compatible structures a Riemann-Poisson structure. In our preceding paper
\cite{aitamrane-zeglaoui}, we generalized this notion to a notion of
compatibility between a Jacobi structure and a Riemannian
structure and proved that in the case of a contact structure and of
a locally conformally symplectic structure we get respectively a
$(1/2)$-Kensmotsu structure and a locally conformally K\"ahler
structure. 

By now, most of these classical geometric structures have been
considered on a Lie algebroid. For a Poisson structure on a
Lie algebroid see for instance \cite{cannas-weinstein,kosmann}. For
a Riemannian structure see \cite{boucetta3}. For almost complex
structures, Hermitian, locally conformally symplectic and locally
conformally K\"ahler structures, see \cite{ida-popescu1} and the
references therein. For contact Riemannian, almost contact
Riemannian and Kenmotsu structures, see \cite{ida-popescu2}.

The goal of the present paper is to generalize to the framework of Lie algebroids the results obtained by the authors on a smooth manifold \cite{aitamrane-zeglaoui}. We first introduce a notion of compatibility  between a Poisson structure and a Riemannian structure on a Lie algebroid that generalizes that same notion introduced by Boucetta on a smooth manifold \cite{boucetta1}, and then we extend this notion of compatibility to a Jacobi structure on a Lie algebroid.

The paper is organized as follows. In section 2,  we recall some basic notions about differential calculus on a Lie
algebroid, \cite{cannas-weinstein,kosmann,marle0}, recall the notion of a Poisson structure on a Lie
algebroid and establish some identities we will need later in the work. Then, we introduce the 
notion of contravariant Levi-Civita connection associated with a bivector
field on a (pseudo-)Riemannian Lie algebroid and a notion of
compatibility of the bivector field and the (pseudo-)Riemannian metric. The Lie algebroid equipped with such a pair of compatible bivector field and metric will be called (pseudo-)Riemannian Poisson Lie algebroid.  

In section 3, we consider a pair of a bivector field and a vector field on a Lie algebroid and 
associate with it a skew algebroid structure on the dual bundle. We compute 
the torsion and the Jacobiator of this skew algebroid; this gives conditions for it to be an almost Lie
algebroid and a Lie algebroid respectively. We investigate the cases of a contact Lie
algebroid and a locally conformally symplectic Lie algebroid. We prove using
tensors that a contact structure is a Jacobi structure that comes from an almost
cosymplectic structure and that a locally conformally symplectic structure
corresponds exactly to a Jacobi structure with a nondegenerate bivector
field, and then we prove that in these two cases the skew algebroid structure 
on the dual bundle is a Lie algebroid structure.

In the last section, we propose a notion of compatibility of a triple
composed of a bivector field, a vector field and a (pseudo-)Riemannian
metric on a skew algebroid. We show under certain conditions that this induces
a Jacobi structure on the skew algebroid. Then, after having recalled 
the notions of contact Riemannian Lie algebroid and (1/2)-Kenmotsu Lie
algebroid, we prove that the triple associated with the contact Riemannian Lie algebroid is
compatible if and only if the Riemannian contact Lie algebroid is (1/2)-Kenmotsu. Finally, we show that in the case of a Jacobi structure
associated with a locally conformally symplectic structure and a somehow
associated metric, the triple is compatible if and only if the locally conformally symplectic structure is locally
conformally K\"{a}hler.

\section{Riemannian Poisson Lie algebroids}

Let $M$ be a smooth manifold. For a vector bundle $A$ over $M$ we will denote by $\Gamma(A)$ the space of sections of $A$.

\subsection{Differential calculus on Lie algebroids}

An anchored vector bundle on $M$ is a vector bundle $A$ on $M$
together with a vector bundle morphism $\rho:A\rightarrow TM$. The
vector bundle morphism $\rho$, called the anchor, induces a map form
$\Gamma(A)$ to $\Gamma(TM)$ that is $C^\infty(M)$-linear. A skew
algebroid on $M$ is an anchored vector bundle $(A,\rho)$ on $M$
endowed with a skew symmetric $\mathbb{R}$-bilinear map
$$
[.,.]: \Gamma(A)\times \Gamma(A) \longrightarrow \Gamma(A)
$$
that satisfies the Leibniz identity
$$
[a,\varphi b]=\varphi [a,b] + (\rho(a)\cdot \varphi)b, \qquad \forall \varphi\in C^\infty(M),\; \forall a,b\in \Gamma(A).
$$
An almost Lie algebroid on the base manifold $M$ is a skew algebroid $(A,\rho,[.,.])$ such that the anchor
$\rho$ satisfies the property :
$$
\rho([a,b])= [\rho(a),\rho(b)], \qquad \forall a,b,\in \Gamma(A).
$$
The bracket on the right hand side of the equality being the usual
Lie bracket. A Lie algebroid on the base manifold $M$ is a skew
algebroid $(A,\rho,[.,.])$ such that the braket $[.,.]$ satisfies
the Jacobi identity :
$$
[[a,b],c]+[[b,c],a]+[[c,a],b]=0, \qquad \forall a,b,c\in \Gamma(A).
$$
It is well known that the Leibniz identity and the Jacobi identity
together imply that the anchor map is a Lie algebra morphism. Hence,
a Lie algebroid is necessarily an almost Lie algebroid.

For a vector bundle $A$ on $M$, we will denote by $A^{\star}$ the
dual vector bundle.  For a non negative integer $k$, denote by
$\Gamma(\wedge^k A)$ (resp. $\Gamma(\wedge^k A^\star)$) the space of
$A$-multivector fields of degree $k$ (resp. the space of $A$-forms
of degree $k$). Denote by $\Gamma(\wedge^\bullet A):=\bigoplus
\Gamma(\wedge^k A)$ the space of $A$-multivector fields and by
$\Gamma(\wedge^\bullet A^\star):=\bigoplus \Gamma(\wedge^k A^\star)$
that of $A$-forms. Equipped with the exterior product "$\wedge$",
they are graded algebras. We can clearly define the interior product
similarly to its definition for differential forms and multivector
fields on a smooth manifold.

Let $(A,\rho,[.,.])$ be a skew algebroid on $M$. We define the Lie
$A$-derivative with respect to a section $a\in \Gamma(A)$ to be the
unique graded endomorphism $\mathcal{L}^\rho_{a}$ of degree $0$ of
the graded algebra $\Gamma(\wedge^\bullet A^\star)$ which satisfies
the following properties :
\begin{equation*}
\mathcal{L}^\rho_{a}\varphi
:=\mathcal{L}_{\rho(a)}\varphi=\rho(a)\cdot \varphi
\end{equation*}
for any $\varphi\in C^\infty(M)$, and, for $k\in \mathbb{N}^{\star
}$,
\begin{equation*}
\mathcal{L}^\rho_{a}\eta (a_{1},\ldots ,a_{k})=\rho(a)\cdot \eta
(a_{1},\ldots ,a_{k})-\sum_{i=1}^{k}\eta (a_{1},\ldots
,[a,a_{i}],\ldots ,a_{k})
\end{equation*}
for any $\eta\in \Gamma(\wedge^k A^\star)$. The Lie $A$-derivative
with respect to $a\in \Gamma(A)$ can be extended to the dual algebra
$\Gamma(\wedge^\bullet A)$ as follows :
\begin{equation*}
\left\langle \mathcal{L}^\rho_{a}P,\eta \right\rangle =\rho (a)\cdot
\left\langle P,\eta \right\rangle -\left\langle
P,\mathcal{L}^\rho_{a}\eta \right\rangle ,
\end{equation*}
where $\left\langle .,.\right\rangle $ is the duality bracket
between $\Gamma(\wedge^\bullet A)$ and $\Gamma(\wedge^\bullet
A^\star)$. We clearly have $\mathcal{L}^\rho_{a}b= \left[
a,b\right]$ for $a,b\in\Gamma(A)$.

Likewise, we define the exterior $A$-differential $d_\rho$,
analoguous to the exterior differential of differential forms,
associated with the skew algebroid $(A,\rho,\left[ .,.\right])$ as
follows : for any $\varphi \in C^{\infty}(M)$,
\begin{equation*}
d_{\rho}\varphi (a)=\mathcal{L}^\rho_{a}\varphi=\rho(a)\cdot\varphi,
\end{equation*}
and, for  $k\geq 1$ and $\eta \in \Gamma(\wedge^k A^\star)$, we set
\begin{equation}\label{A-diff}
\begin{array}{lll}
d_{\rho}\eta (a_{0},\ldots ,a_{k}) & =
&\displaystyle\sum_{i=0}^{k}(-1)^{i}\rho (a_{i})\cdot \eta
(a_{0},\ldots ,\widehat{a_{i}},\ldots
,a_{k}) \\
 & & +\displaystyle\sum_{0\leq i< j\leq k}(-1)^{i+j}\eta
([a_{i},a_{j}],a_{0},\ldots ,\widehat{a_{i}},\ldots ,\widehat{a_{j}},\ldots
,a_{k}).
\end{array}
\end{equation}

The Lie $A$-derivative, the exterior $A$-differential and the
interior product satisfy the Cartan formula, i.e., for $a\in
\Gamma(A)$, we have :
\begin{equation}\label{cartan}
\mathcal{L}^\rho_a=i_a\circ d_\rho + d_\rho \circ i_a.
\end{equation}

Finally, recall that the Schouten-Nijenhuis bracket [.,.] on $A$ is
the unique $\mathbb{R}$-bilinear map $\Gamma(\wedge^\bullet A)\times
\Gamma(\wedge^\bullet A)\rightarrow \Gamma(\wedge^\bullet A)$ which
satisfies the following properties :
\begin{enumerate}
\item[(i)]  for $\varphi,\psi\in C^\infty(M)$,
$$
[\varphi,\psi]=0,
$$
\item[(ii)] for $a\in \Gamma(A)$ and  $P\in \Gamma(\wedge^k A)$,
$$
[a,P]=\mathcal{L}^\rho_a P,
$$
\item[(iii)] for $P\in \Gamma(\wedge^k A)$ and $Q\in \Gamma(\wedge^l A)$;
$$
[P,Q]=(-1)^{kl}[Q,P]
$$
\item[(iv)] for $P\in \Gamma(\wedge^k A)$, $Q\in \Gamma(\wedge^l A)$ and
$R\in \Gamma(\wedge^\bullet A)$,
$$
[P,Q\wedge R]=[P,Q]\wedge R +(-1)^{(k+1)l}Q\wedge [P,R].
$$
\end{enumerate}
Recall also that if the skew algebroid $(A,\rho,[.,.])$ is an almost
Lie algebroid then for any $\varphi\in C^\infty(M)$ we have
$d_\rho(d_\rho \varphi)=0$, and that, the skew algebroid
$(A,\rho,\left[ .,.\right])$ is a Lie algebroid if and only if
$d_\rho \circ d_\rho=0$, which is also equivalent to the
Schouten-Nijenhuis bracket satisfying the graded Jacobi identity
$$
(-1)^{kl}[[Q,R],P]+(-1)^{lr}[[R,P],Q]+(-1)^{rk}[[P,Q],R]=0,
$$
for any $P\in \Gamma(\wedge^k A)$, any $Q\in \Gamma(\wedge^l A)$ and
any $R\in \Gamma(\wedge^r A)$.

\subsection{Poisson Lie algebroids}\label{PoissonLieAlgebroids}

A Poisson skew (resp. almost Lie, Lie) algebroid
$(A,\rho,[.,.],\Pi)$ is a skew (resp. almost Lie, Lie) algebroid
$(A,\rho,[.,.])$ equipped with an $A$-bivector field $\Pi$ that
satisfies $[\Pi,\Pi]=0$. We also say that $\Pi$ is a Poisson
$A$-tensor.

Let $(A,\rho ,\left[ .,.\right] )$ be a skew algebroid on $M$. To
any $A$-bivector field $\Pi $, we associate vector bundle morphisms
$\sharp_\Pi: A^\star\longrightarrow A$ and $\rho_{\Pi}: A^\star
\longrightarrow TM$ defined by
\begin{equation}
\beta \left(\sharp_{\Pi} (\alpha)\right) := \Pi (\alpha, \beta) \qquad \textrm{ and }\qquad \rho_{\Pi}=\rho\circ\sharp_\Pi,
\end{equation}
and a bracket $\left[.,.\right]_\Pi$ on $\Gamma(A^\star)$ defined by
\begin{equation}\label{bracketPi}
[\alpha, \beta]_\Pi := \mathcal{L}^\rho_{\sharp_{\Pi}(\alpha)} \beta
- \mathcal{L}^\rho_{\sharp_{\Pi}(\beta)} \alpha - d_\rho
(\Pi(\alpha, \beta)),
\end{equation}
called the Koszul bracket associated with $\Pi$. Therefore, with a bivector field $\Pi$ on $A$ we
associate a skew algebroid structure $(\rho_\Pi,[.,.]_\Pi)$ on the dual bundle $A^\star$ of $A$.

\begin{prop}
Let $\Pi \in \Gamma(\wedge^2 A)$. For any $Q\in \Gamma(\wedge^k A)$,
we have
$$
d_{\rho_\Pi} Q=-\left[\Pi,Q\right].
$$
\end{prop}
\begin{proof}
The proof is similar to that in the classical case in
\cite[Prop.4.3]{vaisman}.
\end{proof}

In particular with $Q=\Pi$, using the formula (\ref{A-diff}) for the
differential $d_{\rho_\Pi}$ we get the identity
\begin{equation}\label{CrochetPiPi}
[\Pi,\Pi](\alpha,\beta ,\gamma)=-\oint \rho_\Pi(\alpha)\cdot\Pi(\beta,\gamma)+ 
\oint \Pi \left([\alpha,\beta]_\Pi,\gamma \right), 
\end{equation}
where the symbol $\oint$ means the cyclic sum in $\alpha,\beta,\gamma$.  
Using this identity we get the following theorem that gives the
torsion of the skew algebroid $(A^\star,\rho_\Pi,[.,.]_\Pi)$, i.e.
the default for the skew algebroid $(A^\star,\rho_\Pi,[.,.]_\Pi)$ to
be an almost Lie algebroid.

\begin{thm}\label{TorsionPi}
We have the identity
\begin{equation}  \label{torsion-alg-alt-asso-Pi}
\gamma \left(\sharp_\Pi\left( \left[ \alpha ,\beta \right]_{\Pi
}\right) -\left[\sharp_\Pi(\alpha) ,\sharp_\Pi(\beta) \right] \right)
=\frac{1}{2}\left[ \Pi ,\Pi \right](\alpha ,\beta ,\gamma).
\end{equation}
Therefore, if $(A,\rho,[ .,.],\Pi)$ is a Poisson almost Lie
algebroid, the skew algebroid $(A^\star, \rho_\Pi,[.,.]_\Pi)$
associated with is an almost Lie algebroid.
\end{thm}
\begin{proof}
For $\alpha, \beta,\gamma \in \Gamma(A^\star)$, put
$$
\Phi(\alpha,\beta,\gamma):=\gamma\left(\sharp_\Pi([\alpha,\beta]_\Pi)-[\sharp_\Pi(\alpha),\sharp_\Pi(\beta)]\right).
$$
Clearly we have
$\Phi(\alpha,\beta,\gamma)=-\Phi(\beta,\alpha,\gamma)$, i.e. $\Phi$
is skew symmetric in the two first variables. Now, since
$\gamma\left(\sharp_\Pi([\alpha,\beta]_\Pi)\right)
=-[\alpha,\beta]_\Pi\left(\sharp_\Pi(\gamma)\right)$, by
(\ref{bracketPi}) we get
$$
\gamma\left(\sharp_\Pi([\alpha,\beta]_\Pi\right)
=-\mathcal{L}^\rho_{\sharp_{\Pi}(\alpha)}
\beta\left(\sharp_\Pi(\gamma)\right) +
\mathcal{L}^\rho_{\sharp_{\Pi}(\beta)}
\alpha\left(\sharp_\Pi(\gamma)\right) + d_\rho (\Pi(\alpha,
\beta))\left(\sharp_\Pi(\gamma)\right),
$$
and therefore
$$
\begin{array}{lll}
\Phi(\alpha,\beta,\gamma)&=&\rho_\Pi(\alpha)\cdot\Pi(\beta,\gamma)+\rho_\Pi(\beta)\cdot\Pi(\gamma,\alpha)
+\rho_\Pi(\gamma)\cdot\Pi(\alpha,\beta) \\ &&
-\alpha\left([\sharp_\Pi(\beta),\sharp_\Pi(\gamma)]\right)
-\beta\left([\sharp_\Pi(\gamma),\sharp_\Pi(\alpha)]\right)
-\gamma\left([\sharp_\Pi(\alpha),\sharp_\Pi(\beta)]\right)
\end{array}
$$
from which we deduce, on one hand, that
$\Phi(\alpha,\gamma,\beta)=-\Phi(\alpha,\beta,\gamma)$, and on the
other hand, using the identity (\ref{CrochetPiPi}) and the
definition of the mapping $\Phi$, that
$$
\Phi(\alpha,\beta,\gamma)=-[\Pi,\Pi](\alpha,\beta,\gamma)+\Phi(\alpha,\beta,\gamma)
+\Phi(\beta,\gamma,\alpha)+\Phi(\gamma,\alpha,\beta).
$$
Hence, since $\Phi$ is an alternating map,
$[\Pi,\Pi](\alpha,\beta,\gamma)=2\Phi(\alpha,\beta,\gamma)$.
 \end{proof}

\begin{rem}\label{symp-Lie-algebroid}
A symplectic form on a skew (resp. almost Lie, Lie) algebroid
$(A,\rho,[.,.])$, or a symplectic $A$-form, is a nondegenerate
$A$-$2$-form $\Omega\in \Gamma(\wedge^2 A^\star)$ such that
$d_{\rho}\Omega =0$. We call $(A,\rho,[.,.],\Omega)$ a symplectic
skew (resp. almost Lie, Lie) algebroid. Let $(A,\rho,[.,.])$ be a
skew algebroid, $\Omega\in\Gamma(\wedge^2 A^\star)$ be a
nondegenerate $2$-form on $A$, $\sharp_\Omega$ be the inverse
isomorphism of the vector bundle isomorphism $\flat_\Omega:
A\rightarrow A^\star$, $\flat_\Omega(a)=-i_a\Omega$, and $\Pi$ be
the $A$-bivector field defined by
$$
\Pi(\alpha,\beta)=\Omega(\sharp_{\Omega}(\alpha),\sharp_\Omega(\beta)).
$$
We have $\sharp_\Pi=\sharp_\Omega$, and using the formulae
(\ref{CrochetPiPi}) and (\ref{torsion-alg-alt-asso-Pi}) we get the
identity
$$
[\Pi,\Pi](\alpha,\beta,\gamma)=2d_\rho\Omega(\sharp_\Omega(\alpha),\sharp_\Omega(\beta),\sharp_\Omega(\gamma)).
$$
Hence, $\Pi$ is a Poisson $A$-tensor if and only if $\Omega$ is a
symplectic $A$-form.
\end{rem}

Denote by $J_\Pi$ the Jacobiator of the skew algebroid
$(A^\star,\rho_\Pi,[.,.]_\Pi)$, i.e., 
$$
J_\Pi(\alpha,\beta,\gamma)=\oint [[\alpha,\beta]_\Pi,\gamma]_\Pi, 
$$  
for $\alpha,\beta,\gamma\in \Gamma(A^\star)$. The following result gives for a Lie algebroid $(A,\rho,[.,.])$ 
together with an $A$-bivector field $\Pi$ the Jacobiator 
of the skew algebroid $(A^\star,\rho_\Pi,[.,.]_\Pi)$, i.e. the default for  
$(A^\star,\rho_\Pi,[.,.]_\Pi)$ to be a Lie algebroid.  

\begin{thm}\label{JacobiatorPi}
Assume that $(A,\rho,\Pi)$  is a Lie algebroid. Then
$$
J_\Pi(\alpha,\beta,\gamma)=\oint \mathcal{L}^\rho_{\sharp_\Pi([\alpha,\beta]_\Pi)-[\sharp_\Pi(\alpha),\sharp_\Pi(\beta)]}\gamma-\oint d_\rho\left(\Pi(\mathcal{L}^\rho_{\sharp_\Pi(\alpha)}\beta,\gamma) +\Pi(\beta,\mathcal{L}^\rho_{\sharp_\Pi(\alpha)}\gamma)\right). 
$$
\end{thm}
\begin{proof}
First, by (\ref{bracketPi}), for $\alpha,\beta,\gamma \in\Gamma(A^\star)$, we have 
$$
[d_\rho(\Pi(\alpha,\beta)),\gamma]_\Pi=\mathcal{L}^\rho_{\sharp_\Pi(d_\rho(\Pi(\alpha,\beta)))}\gamma-\mathcal{L}^\rho_{\sharp_\Pi(\gamma)}\left(d_\rho(\Pi(\alpha,\beta))\right)-d_\rho\left(\Pi(d_\rho(\Pi(\alpha,\beta)),\gamma)\right).
$$
Since $(A,\rho,[.,.])$ is a Lie algebroid, hence an almost Lie algebroid, we have 
$$
\mathcal{L}^\rho_{\sharp_\Pi(\gamma)}\left(d_\rho(\Pi(\alpha,\beta))\right)=d_\rho(\mathcal{L}^\rho_{\sharp_\Pi(\gamma)}\left(\Pi(\alpha,\beta)\right))=d_\rho(\rho_\Pi(\gamma)\cdot \Pi(\alpha,\beta)), 
$$
and since 
$$
\Pi(d_\rho(\Pi(\alpha,\beta)),\gamma)=-d_\rho(\Pi(\alpha,\beta))(\sharp_\Pi(\gamma))=-\rho_\Pi(\gamma)\cdot \Pi(\alpha,\beta),
$$
we deduce that 
$$
[d_\rho(\Pi(\alpha,\beta)),\gamma]_\Pi=\mathcal{L}^\rho_{\sharp_\Pi(d_\rho(\Pi(\alpha,\beta)))}\gamma.
$$
Consequently, by (\ref{bracketPi}), we get   
\begin{equation}\label{bracketPiPi}
\begin{array}{ll}
[[\alpha,\beta]_\Pi,\gamma]_\Pi = &
\mathcal{L}^\rho_{\sharp_\Pi([\alpha,\beta]_\Pi)}\gamma
-\mathcal{L}^\rho_{\sharp_\Pi(\gamma)}\left(\mathcal{L}^\rho_{\sharp_\Pi(\alpha)}\beta\right)
+\mathcal{L}^\rho_{\sharp_\Pi(\gamma)}\left(\mathcal{L}^\rho_{\sharp_\Pi(\beta)}\alpha\right)
\\ &
-d_\rho\left[\Pi(\mathcal{L}^\rho_{\sharp_\Pi(\alpha)}\beta,\gamma)
+\Pi(\gamma,\mathcal{L}^\rho_{\sharp_\Pi(\beta)}\alpha)\right]. 
\end{array}
\end{equation}
Since $(A,\rho,[.,.])$ is a Lie algebroid we have
$$
\mathcal{L}^\rho_{\sharp_\Pi(\gamma)}\left(\mathcal{L}^\rho_{\sharp_\Pi(\alpha)}\beta\right)
-\mathcal{L}^\rho_{\sharp_\Pi(\alpha)}\left(\mathcal{L}^\rho_{\sharp_\Pi(\gamma)}\beta\right)
=\mathcal{L}^\rho_{[\sharp_\Pi(\gamma),\sharp_\Pi(\alpha)]}\beta.
$$
Therefore, taking the cyclic sum in $\alpha,\beta,\gamma$ on the two sides of the equality (\ref{bracketPiPi}) we get the result. 
\end{proof}

\begin{cor}
If $(A,\rho,[.,.],\Pi)$ is a Poisson Lie algebroid, then
$(A^\star,\rho_\Pi,[.,.]_\Pi)$ is a Lie algebroid.
\end{cor} 
\begin{proof}
We need to prove that $J_\Pi=0$. Since $[\Pi,\Pi]=0$, by Theorem \ref{TorsionPi} and the theorem above we have  
$$
J_\Pi(\alpha,\beta,\gamma)=-\oint d_\rho\left(\Pi(\mathcal{L}^\rho_{\sharp_\Pi(\alpha)}\beta,\gamma) +\Pi(\beta,\mathcal{L}^\rho_{\sharp_\Pi(\alpha)}\gamma)\right). 
$$ 
For any $\alpha,\beta,\gamma$ we have 
$$
\begin{array}{ll}
\Pi(\mathcal{L}^\rho_{\sharp_\Pi(\alpha)}\beta,\gamma) & =-\mathcal{L}^\rho_{\sharp_\Pi(\alpha)}\beta(\sharp_\Pi(\gamma))\\
 & =-\rho_\Pi(\alpha)(\beta(\sharp_\Pi(\gamma)))+\beta([\sharp_\Pi(\alpha),\sharp_\Pi(\gamma)]) \\
 & =\rho_\Pi(\alpha)\cdot \Pi(\beta,\gamma)-\beta([\sharp_\Pi(\gamma),\sharp_\Pi(\alpha)]).
\end{array}
$$
Substituting in the expression of $J_\Pi$, we get 
$$
J_\Pi(\alpha,\beta,\gamma)= -2d_\rho\left[\oint \rho_\Pi(\alpha)\cdot \Pi(\beta,\gamma) 
-\oint \alpha([\sharp_\Pi(\beta),\sharp_\Pi(\gamma)])\right].
$$
Since $[\Pi,\Pi]$=0, by (\ref{CrochetPiPi}) and Theorem \ref{TorsionPi}, it follows that $J_\Pi=0$. 
\end{proof}

When $A$ is the tangent Lie algebroid $TM$, the Lie algebroid
$(A^\star,\rho_\Pi,[.,.]_\Pi)$ is the cotangent Lie algebroid of the
Poisson manifold $(M,\Pi)$.

\subsection{Riemannian Poisson Lie algebroids}

\subsubsection{Contravariant $A$-connections associated with an $A$-bivector field}\label{contravariant-connection-Pi-g}

Let $(A,\rho,[.,.])$ be a skew algebroid. An affine connection on  $(A,\rho,[.,.])$,
or an affine $A$-connection, is an $\mathbb{R}$-bilinear map
$$
\begin{array}{rccc}
\nabla : &  \Gamma(A) \times \Gamma(A) & \longrightarrow & \Gamma(A) \\
& (a,b) & \mapsto & \nabla_a b
\end{array}
$$
verifying $\nabla_{\varphi a }b =\varphi \nabla_{a}b$ and
$\nabla_{a}(\varphi b)=\varphi \nabla_{a}b +\left( \rho(a).\varphi
\right) b$, for any $\varphi \in C^{\infty}(M) $ and $a,b \in
\Gamma(A)$. The $A$-torsion $T^{\nabla}$ of the connection $ \nabla$
is defined by
\begin{equation*}
T^{\nabla}(a,b) =\nabla_{a}b -\nabla_{b}a - [a,b]
\end{equation*}
for any $a,b\in \Gamma(A)$. The $A$-connection $\nabla$ is called
symmetric, or torsion free, if its $A$-torsion is null.

Let $\Pi$ be an $A$-bivector field. An affine connection on the skew
algebroid $(A^\star,\rho_\Pi,[.,.]_\Pi)$ associated with $\Pi$ will
be called a contravariant $A$-connection associated with $\Pi$. If
$\Pi $ is nondegenerate, the formula
$$
D_\alpha \beta = {\sharp_\Pi}^{-1}
\left(\nabla_{\sharp_\Pi(\alpha)}(\sharp_\Pi(\beta)) \right)
$$
establishes a one-one correspondence between the affine $A$-connections $\nabla$ and the
contravariant $A$-connections $D$ associated with the $A$-bivector field $\Pi$.

A (pseudo-)Riemannian skew (resp. almost Lie, Lie) algebroid is a quadruple $(A,\rho,[ .,.],g)$
where $(A,\rho,[.,.])$ is a skew (resp. almost Lie, Lie) algebroid and $g$ is a
(pseudo-)Riemannian metric on its underlying vector bundle $A$.

Let $(A,\rho,[.,.],g)$ be a pseudo-Riemannian skew algebroid. There
is a unique affine $A$-connection $\nabla$ which is symmetric, i.e.
$$
\nabla_a b - \nabla_b a=[a,b],
$$
and compatible with $g$, i.e.
$$
\rho(a) \cdot g(b,c)=g\left(\nabla_a b,c\right)+g\left(b,\nabla_a
c\right).
$$
It is entirely characterized by the Koszul formula :
\begin{equation}\label{koszul}
\begin{array}{lll}
2g\left(\nabla_a b ,c \right) \!\!&=&\!\! \rho(a)\cdot g(b,c)
+ \rho(b)\cdot g(a,c) - \rho(c)\cdot g(a,b)  \\
&& -g\left( [b,c],a \right)-g\left( [a,c],b \right)+g\left( [a,b],c \right) ,
\end{array}
\end{equation}
for any $a ,b,c \in \Gamma(A)$. We call $\nabla$ the Levi-Civita
$A$-connection associated with $g$, or the Levi-Civita connection of
the pseudo-Riemannian skew algebroid  $\left(A,\rho,[.,.],g\right)$.

Let $\flat_{g}$ be the vector bundle isomorphism $\Gamma(A)
\rightarrow \Gamma(A^{\star})$, $a \longmapsto g(a,.)$, and
$\sharp_g$ be the inverse isomorphism. The cometric $g^\star$ of $g$
is the pseudo-Riemannian metric on the dual bundle $A^\star$ defined
by
\begin{equation*}
g^\star(\alpha,\beta) = g(\sharp_g(\alpha),\sharp_g(\beta)), \qquad \alpha,\beta\in\Gamma(A^\star).
\end{equation*}
Let $\Pi$ be a bivector field on $A$ and let
$\left(A^\star,\rho_\Pi,[ .,.]_\Pi\right)$ be the associated skew
algebroid. The Levi-Civita connection $D$ of the pseudo-Riemannian
skew algebroid $(A^\star,\rho_\Pi,[.,.]_\Pi,g^\star)$ will be called
the contravariant Levi-Civita $A$-connection associated with the
pair $(\Pi,g)$.

\subsubsection{Compatibility of a bivector field and a 
Riemannian metric on a Lie algebroid} \label{Poisson-Riemann}

Let $(A,\rho,[.,.],g)$ be a pseudo-Riemannian skew algebroid and let
$\Pi$ ba a bivector field on $A$. We say that the metric $g$ is
compatible with the bivector field $\Pi$ or that the pair $(\Pi ,g)$
is compatible if $D\Pi =0$, i.e., if
\begin{equation*}
\rho_{\Pi }(\alpha )\cdot\Pi (\beta ,\gamma )=\Pi (D_{\alpha }\beta
,\gamma )+\Pi (\beta ,D_{\alpha }\gamma ),
\end{equation*}
for every $\alpha ,\beta ,\gamma \in \Gamma (A^{\star})$.

With the pair $(\Pi,g)$ we associate the vector bundle
endomorphisms $J$ of $A$ and $J^{\star}$ of $A^\star$ defined
respectively by
\begin{equation}\label{J et J star}
g\left(J\sharp_g(\alpha),\sharp_g(\beta)\right)=\Pi \left(\alpha,\beta \right)
\qquad\textrm{ and }\qquad g^\star\left(J^\star\alpha,\beta\right)=\Pi\left(\alpha,\beta\right).
\end{equation}
Since $Dg^\star=0$, we have, for any $\alpha,\beta,\gamma\in \Gamma
(A^{\star})$,
$$
\begin{array}{ll}
g^\star(\alpha,D_\gamma J^\star(\beta)) &=g^\star(\alpha,D_\gamma
(J^\star\beta))-g^\star(\alpha,J^\star(D_\gamma\beta)) \\
& =\rho_{\Pi}(\gamma)\cdot
g^\star(\alpha,J^\star\beta)-g^\star(D_\gamma\alpha,J^\star\beta)
-g^\star(\alpha,J^\star(D_\gamma\beta)),
\end{array}
$$
and therefore, by the second equality in (\ref{J et J star}),
\begin{equation}\label{DJDPi}
g^\star(\alpha,D_\gamma J^\star(\beta))=-\rho_{\Pi}\left( \gamma
\right)\cdot \Pi \left(\alpha,\beta\right)+\Pi \left(D_\gamma
\alpha,\beta \right) +\Pi \left( \alpha ,D_{\gamma }\beta \right).
\end{equation}
Hence $g^\star(\alpha,D_\gamma
J^\star(\beta))=-D_\gamma\Pi(\alpha,\beta)$, from which we deduce
that the pair $(\Pi ,g) $ is compatible if and only if $DJ^{\star
}=0$, or equivalently $DJ=0$.

\begin{prop}\label{bivector-compatibility-Poisson}
Let $\left(A,\rho,[ .,.] ,g\right)$ be a pseudo-Riemannian skew
algebroid and let $\Pi$ be an $A$-bivector field. If the pair $(\Pi
,g)$ is compatible then $\left(A,\rho ,[.,.] ,\Pi \right) $ is a
Poisson skew algebroid.
\end{prop}
\begin{proof}
Let $\alpha ,\beta ,\gamma \in \Gamma \left(A^{\star }\right)$.
Taking the cyclic sum in $\alpha,\beta,\gamma$ on the two sides of
the formula (\ref{DJDPi}), taking in account that $D$ is symmetric
and using (\ref{CrochetPiPi}) we get the identity :
\begin{equation*}
\left[ \Pi ,\Pi \right] \left( \alpha ,\beta ,\gamma \right)
=g^{\star }\left( \alpha ,D_{\gamma }J^{\star }\left( \beta \right)
\right) +g^{\star }\left( \beta ,D_{\alpha }J^{\star }\left( \gamma
\right) \right) +g^{\star }\left( \gamma ,D_{\beta }J^{\star }\left(
\alpha \right) \right)
\end{equation*}
from which the result follows.
\end{proof}

A (pseudo-)Riemannian Poisson skew (resp. almost Lie, Lie) algebroid
$\left(A,\rho,[.,.],\Pi,g\right) $  is a skew (resp. almost Lie,
Lie) algebroid $\left( A,\rho,[.,.]\right) $ equipped with a
 (pseudo-)Riemannian metric $g$ and an $A$-bivector field $\Pi$ that
 are compatible. This definition is justified by the proposition above (i.e. $\Pi$
 is necessarily Poisson).

\begin{ex}
\begin{enumerate}
\item When $A=TM$ , we get the notion of (pseudo-)Riemannian Poisson
manifold defined by M. Boucetta, see \cite{boucetta1,boucetta2}.
\item Let $(A,\rho,[.,.])$ be a skew algebroid, $\Omega$ be
a nondegenerate $A$-form and $\Pi$ be the $A$-bivector field
associated with in Remark \ref{symp-Lie-algebroid}. Assume that $g$
is a pseudo-Riemannian metric on $A$ associated with $\Omega$, i.e.,
$$
g^\star(\alpha,\beta)=g(\sharp_\Omega(\alpha),
\sharp_\Omega(\beta)),
$$
for any $\alpha,\beta\in\Gamma(A^\star)$. If $D$ is the
contravariant Levi-Civita $A$-connection associated with the pair
$(\Pi,g)$, then
$$
\sharp_\Omega(D_\alpha
\beta)=\nabla_{\sharp_\Omega(\alpha)}\left(\sharp_\Omega(\beta)\right),
$$
where $\nabla$ is the (covariant) Levi-Civita $A$-connection associated
with $g$. Hence,
$$
D\Pi(\alpha,\beta,\gamma)=\nabla\Omega(\sharp_\Omega(\alpha),\sharp_\Omega(\beta),\sharp_\Omega(\gamma))
$$
for any $\alpha,\beta,\gamma\in \Gamma(A^\star)$. Therefore,
$(A,\rho,[.,.],\Pi,g)$ is Riemannian Poisson if and only if
$(A,\rho,[.,.],\Omega,g)$ is K\"ahler.

\item Let $\left( \mathfrak{g},[.,.]\right)$ be a real Lie algebra of
finite dimension seen as a Lie algebroid over a point ($\rho=0$). A
$\mathfrak{g}$-bivector field $\Pi$ is just a skew symmetric
bilinear form on the dual space $\mathfrak{g}^\star$ of
$\mathfrak{g}$. Hence, the associated Koszul bracket $[.,.]_\Pi$ is
defined by
$$
[\alpha,\beta]_\Pi(a)=\alpha\left([\sharp_\Pi(\beta),a]\right)-\beta\left([\sharp_\Pi(\alpha),a]\right),
$$
for any $\alpha,\beta\in \mathfrak{g}^\star$ and any $a\in
\mathfrak{g}$. Let $\langle .,. \rangle$ be a (pseudo-)Riemannian
metric on the Lie algebroid $\left( \mathfrak{g},[.,.]\right)$, i.e.
a scalar product on the vector space $\mathfrak{g}$, and let
$\langle .,. \rangle^\star$ be its cometric. The contravariant
Levi-Civita $\mathfrak{g}$-connection associated with the pair
$\left(\Pi, \langle .,. \rangle\right)$ is defined by
$$
2\left\langle D_{\alpha }\beta ,\gamma
\right\rangle^\star=\left\langle [\alpha,\beta] _{\Pi },\gamma
\right\rangle ^{\star }-\left\langle [\beta,\gamma]_\Pi,\alpha
\right\rangle^\star-\left\langle [\alpha,\gamma]_\Pi,\beta
\right\rangle^\star.
$$
The pair $\left(\Pi, \langle .,. \rangle\right)$ is compatible if
and only if
$$
\Pi\left(D_\alpha\beta,\gamma\right)+\Pi\left(\beta,D_\alpha\gamma\right)=0,
$$
for every $\alpha,\beta,\gamma\in \mathfrak{g}^\star$. So, if this
last identity is satisfied, we call the quadruple
$\left(\mathfrak{g},[.,.],\Pi,\langle .,. \rangle \right)$ a
(pseudo-)Riemannian Poisson Lie algebra.
\end{enumerate}
\end{ex}

\section{Jacobi structure on a Lie algebroid}

\subsection{Jacobi structure on an almost Lie
algebroid}

Let $(A,\rho,[.,.])$ be a skew algebroid on $M$. Let $\Pi$ be an
$A$-bivector field and $\xi$ be an $A$-vector field. We say that the
pair $(\Pi,\xi)$ is a Jacobi structure on $\left(A,\rho ,\left[
.,.\right]\right) $ if the following identities are satisfied
$$
\left[ \Pi ,\Pi \right] =2\xi \wedge \Pi \qquad \textrm{ and }
\qquad \left[ \xi ,\Pi \right] =\mathcal{L}^\rho_\xi\Pi=0.
$$
A Jacobi structure $(\Pi,\xi)$ for which $\xi=0$ is just a Poisson
structure.

If $(A,\rho,[.,.])$ is an almost Lie algebroid the pushforward of
multivectors $\rho_\star :\Gamma(\wedge^\bullet A)\rightarrow
\Gamma(\wedge^\bullet TM)$ is compatible with the
Schouten-Nijenhuis bracket and in particular we have
$$
\rho_\star\left[\Pi,\Pi\right]=\left[\rho_\star\Pi,\rho_\star\Pi\right]  \qquad \textrm{ and }
\qquad \rho_\star\left[\xi ,\Pi \right]=\left[ \rho_\star\xi ,\rho_\star \Pi \right].
$$
Therefore, a Jacobi structure $(\Pi,\xi)$ on the almost Lie
algebroid $(A,\rho,[.,.])$ induces a Jacobi structure
$(\rho_\star\Pi,\rho_\star\xi)$ on the manifold $M$. The
corresponding Jacobi bracket on the smooth functions $C^\infty(M)$
is given by
$$
\left\{\varphi,\psi\right\}=\rho_\star\Pi\left(d\varphi,d\psi\right)+\varphi\mathcal{L}_{\rho_\star\xi}\psi-\psi\mathcal{L}_{\rho_\star\xi}\varphi,
$$
i.e.
$$
\left\{\varphi,\psi\right\}=\Pi\left(d_\rho\varphi,d_\rho\psi\right)+\varphi\mathcal{L}^\rho_\xi\psi-\psi\mathcal{L}^\rho_\xi\varphi.
$$
When $\xi=0$, the $A$-Poisson tensor $\Pi$ induces a Poisson tensor
$\rho_\star\Pi$ on $M$, the corresponding Poisson bracket on the
smooth functions $C^\infty(M)$ being given by
$\left\{\varphi,\psi\right\}=\Pi(d_\rho\varphi,d_\rho\psi)$.

\subsection{Skew algebroids associated with the pair $(\Pi,\xi)$}

Let $\left(A,\rho,\left[ .,.\right]\right)$ be a skew algebroid. Let
$\Pi$ be an $A$-bivector field and $\xi$ be an $A$-vector field.
With the pair $(\Pi,\xi)$ we associate the vector bundle morphisms
$\sharp_{\Pi,\xi}:A^\star \rightarrow A$ and $\rho_{\Pi,\xi}:A^\star
\rightarrow TM$ defined by
\begin{equation}\label{rhoPixi}
\sharp_{\Pi,\xi}(\alpha)=\sharp_{\Pi}(\alpha) +\alpha(\xi)\xi \qquad \textrm{ and } \qquad \rho_{\Pi,\xi}=\rho\circ \sharp_{\Pi,\xi}.
\end{equation}
For any section $\lambda$ of the vector bundle $A^{\star }$,
consider the bracket $\left[.,.\right]_{\Pi,\xi}^\lambda$ on
$\Gamma(A^\star)$ defined by
\begin{equation}\label{bracketPixi}
\left[ \alpha ,\beta \right] _{\Pi ,\xi }^{\lambda }
=\left[\alpha,\beta\right]_{\Pi}+ \alpha(\xi) \left( \mathcal{L}^\rho_\xi\beta -\beta \right) -\beta(\xi) \left( \mathcal{L}^\rho_\xi\alpha -\alpha \right)
-\Pi(\alpha,\beta)\lambda.
\end{equation}
The pair $\left(\rho_{\Pi ,\xi },\left[ .,.\right] _{\Pi ,\xi
}^{\lambda }\right)$ induces a skew algebroid structure on the dual
bundle of $A$. We call $\left( A^\star,\rho_{\Pi ,\xi },\left[
.,.\right] _{\Pi ,\xi }^{\lambda }\right)$ the skew algebroid
associated with the triple $\left( \Pi ,\xi ,\lambda \right) $. In
the particular case where $\xi =0$ and $\lambda =0$, we get the skew
algebroid $\left(A^{\star },\rho_{\Pi },\left[ .,.\right] _{\Pi
}\right) $ associated in \S\ref{PoissonLieAlgebroids} with the
$A$-bivector field $\Pi$.

\begin{thm}\label{alg-alt-jacobiO}
Assume that $\left( \Pi ,\xi \right)$ is a Jacobi structure on the
skew algebroid $\left( A,\rho,[.,.] \right)$ and let $\lambda \in
\Gamma(A^\star)$. We have
\begin{equation*}
\sharp _{\Pi ,\xi }\left( \left[ \alpha ,\beta \right]_{\Pi
,\xi}^{\lambda }\right) -\left[ \sharp_{\Pi,\xi}(\alpha)
,\sharp_{\Pi ,\xi}(\beta) \right]=\Pi(\alpha ,\beta)\left(\xi
-\sharp _{\Pi ,\xi }(\lambda) \right),
\end{equation*}
for any $\alpha ,\beta \in \Gamma(A^\star)$.
\end{thm}
\begin{proof}
The proof is the same as in the classical case of the tangent
algebroid $A=TM$, \cite[Th. 2.1]{aitamrane-zeglaoui}. Use Theorem
\ref{TorsionPi}.
\end{proof}

\begin{cor}\label{alg-alt-jacobi}
Assume that $(A,\rho,[.,.])$ in the theorem above is an almost Lie
algebroid. If $\xi -\sharp _{\Pi ,\xi }(\lambda)\in \ker\rho$, then
the skew algebroid $\left( A^\star,\rho_{\Pi,\xi},\left[
.,.\right]_{\Pi ,\xi}^{\lambda}\right)$ is an almost Lie algebroid,
i.e., we have
\begin{equation*}
\rho_{\Pi,\xi}\left( \left[ \alpha ,\beta \right] _{\Pi ,\xi
}^{\lambda }\right) =\left[ \rho_{\Pi ,\xi }(\alpha)
,\rho_{\Pi,\xi}(\beta) \right],
\end{equation*}%
for any  $A$-forms $\alpha $ and $\beta$. The converse is also true
if $\Pi \neq 0$.
\end{cor}
\begin{proof}
Applying $\rho$ to the identity in the theorem above, we get the
following identity
\begin{equation*}
\rho_{\Pi ,\xi }\left([\alpha ,\beta]_{\Pi,\xi}^{\lambda}\right)
-\left[ \rho_{\Pi,\xi}(\alpha) ,\rho_{\Pi ,\xi}(\beta)
\right]=\Pi(\alpha ,\beta)\rho\left(\xi -\sharp _{\Pi ,\xi
}(\lambda) \right)
\end{equation*}
that gives the torsion of the skew algebroid $\left(
A^\star,\rho_{\Pi ,\xi },\left[ .,.\right] _{\Pi ,\xi }^{\lambda
}\right)$, i.e. the default for it to be an almost Lie algebroid.
\end{proof}

The following result gives for a Lie algebroid $(A,\rho,[.,.])$ the Jacobiator $J^\lambda_{\Pi,\xi}$, 
$$
J^\lambda_{\Pi,\xi}(\alpha,\beta,\gamma):=\oint [[\alpha,\beta]^\lambda_{\Pi,\xi},\gamma]^\lambda_{\Pi,\xi}, 
$$
of the skew algebroid $(A^\star,\rho_{\Pi,\xi},[.,.]^\lambda_{\Pi,\xi})$.

\begin{thm} 
Assume that $(A,\rho,[.,.])$ is a Lie algebroid. We have
$$
\begin{array}{ll}
J_{\Pi,\xi}^\lambda(\alpha,\beta,\gamma)= & \oint \mathcal{L}^\rho_{(\frac{1}{2}[\Pi,\Pi]-\xi\wedge \Pi)(\alpha,\beta,\cdot)}\gamma  +\left[(2\xi\wedge \Pi -[\Pi,\Pi])(\alpha,\beta,\gamma)+\oint\gamma(\xi)\mathcal{L}^\rho_\xi\Pi(\alpha,\beta)\right]\lambda \\  
& + d_\rho\left((2\xi\wedge \Pi-[\Pi,\Pi])(\alpha,\beta,\gamma)\right)-\oint\gamma(\xi)\mathcal{L}^\rho_\xi([.,.]_\Pi)(\alpha,\beta)  \\
& +\oint(\mathcal{L}^\rho_\xi\gamma-\gamma)\mathcal{L}^\rho_\xi\Pi(\alpha,\beta)+ \oint \Pi(\alpha,\beta)\left(\mathcal{L}^\rho_{\xi}\gamma
-[\lambda,\gamma]^\lambda_{\Pi,\xi}\right),
\end{array}
$$
where $\mathcal{L}^\rho_\xi([.,.]_\Pi)(\alpha,\beta)=\mathcal{L}^\rho_\xi([\alpha,\beta]_\Pi)-[\mathcal{L}^\rho_\xi\alpha,\beta]_\Pi-[\alpha,\mathcal{L}^\rho_\xi\beta]_\Pi$.
\end{thm}

\begin{proof}
A long but direct calculation using (\ref{bracketPi}),(\ref{CrochetPiPi}),(\ref{rhoPixi}) and (\ref{bracketPixi}) gives 
$$
\begin{array}{ll}
J^\lambda_{\Pi,\xi}(\alpha,\beta,\gamma)= & J_\Pi(\alpha,\beta,\gamma)-\oint\gamma(\xi)\mathcal{L}^\rho_\xi([.,.]_\Pi)(\alpha,\beta)
+\oint(\mathcal{L}^\rho_\xi\gamma-\gamma)\mathcal{L}^\rho_\xi\Pi(\alpha,\beta) \\ 
 & -\oint\gamma(\xi)[\alpha,\beta]_\Pi +\left((2\xi\wedge\Pi-[\Pi,\Pi])(\alpha,\beta,\gamma)+\oint\gamma(\xi)\mathcal{L}^\rho_\xi\Pi(\alpha,\beta)\right)\lambda \\ 
  & +\oint\Pi(\alpha,\beta)[\gamma,\lambda]^\lambda_{\Pi,\xi}. 
\end{array}
$$
By Theorem \ref{JacobiatorPi}, we have 
$$
J_\Pi(\alpha,\beta,\gamma)=\oint \mathcal{L}^\rho_{\sharp_\Pi([\alpha,\beta]_\Pi)-[\sharp_\Pi(\alpha),\sharp_\Pi(\beta)]}\gamma-\oint d_\rho\left(\Pi(\mathcal{L}^\rho_{\sharp_\Pi(\alpha)}\beta,\gamma) +\Pi(\beta,\mathcal{L}^\rho_{\sharp_\Pi(\alpha)}\gamma)\right). 
$$
On one hand, since by (\ref{torsion-alg-alt-asso-Pi}) we have  
$$
\begin{array}{ll}
\sharp_\Pi([\alpha,\beta]_\Pi)-[\sharp_\Pi(\alpha),\sharp_\Pi(\beta)]= & \left(\frac{1}{2}[\Pi,\Pi]-\xi\wedge\Pi\right)(\alpha,\beta,.)+\alpha(\xi)\sharp_\Pi(\beta)-\beta(\xi)\sharp_\Pi(\alpha) \\ 
 & +\Pi(\alpha,\beta)\xi, 
\end{array} 
$$
it comes that 
$$
\begin{array}{ll}
\mathcal{L}^\rho_{\sharp_\Pi([\alpha,\beta]_\Pi)-[\sharp_\Pi(\alpha),\sharp_\Pi(\beta)]}\gamma = & \mathcal{L}^\rho_{\left(\frac{1}{2}[\Pi,\Pi]-\xi\wedge\Pi\right)(\alpha,\beta,.)}\gamma + \alpha(\xi)\mathcal{L}^\rho_{\sharp_\Pi(\beta)}\gamma -\beta(\xi)\mathcal{L}^\rho_{\sharp_\Pi(\alpha)}\gamma \\ 
 & +\gamma(\xi)d_\rho(\Pi(\alpha,\beta))+\Pi(\beta,\gamma)d_\rho(\alpha(\xi))+\Pi(\gamma,\alpha)d_\rho(\beta(\xi)) \\ 
 & +\Pi(\alpha,\beta) \mathcal{L}^\rho_\xi\gamma. 
\end{array} 
$$
Taking the cyclic sum in $\alpha,\beta,\gamma$ on the two sides we get 
$$
\begin{array}{ll}
\oint\mathcal{L}^\rho_{\sharp_\Pi([\alpha,\beta]_\Pi)-[\sharp_\Pi(\alpha),\sharp_\Pi(\beta)]}\gamma = & 
\oint \mathcal{L}^\rho_{\left(\frac{1}{2}[\Pi,\Pi]-\xi\wedge\Pi\right)(\alpha,\beta,.)}\gamma +\oint \gamma(\xi)[\alpha,\beta]_\Pi \\ 
 & +2d_\rho\left((\xi\wedge\Pi)(\alpha,\beta,\gamma)\right) + \oint \Pi(\alpha,\beta)\mathcal{L}^\rho_\xi\gamma. 
\end{array}
$$
On the other hand, we have 
$$
\Pi(\mathcal{L}^\rho_{\sharp_\Pi(\alpha)}\beta,\gamma)=-\mathcal{L}^\rho_{\sharp_\Pi(\alpha)}\beta(\sharp_\Pi(\gamma))
=\rho_\Pi(\alpha)\cdot \Pi(\beta,\gamma)-\beta([\sharp_\Pi(\gamma),\sharp_\Pi(\alpha)]), 
$$
and then, by (\ref{torsion-alg-alt-asso-Pi}), 
$$
\Pi(\mathcal{L}^\rho_{\sharp_\Pi(\alpha)}\beta,\gamma)=\rho_\Pi(\alpha)\cdot \Pi(\beta,\gamma)
-\Pi([\gamma,\alpha]_\Pi,\beta)+\frac{1}{2}[\Pi,\Pi](\gamma,\alpha,\beta). 
$$
Hence,
$$
\begin{array}{ll}
\Pi(\mathcal{L}^\rho_{\sharp_\Pi(\alpha)}\beta,\gamma)+\Pi(\beta,\mathcal{L}^\rho_{\sharp_\Pi(\alpha)}\gamma)= & 2\rho_\Pi(\alpha)\cdot \Pi(\beta,\gamma) -\Pi([\gamma,\alpha]_\Pi,\beta)-\Pi([\alpha,\beta]_\Pi,\gamma)\\ 
& +[\Pi,\Pi](\alpha,\beta,\gamma),
\end{array}
$$
and taking the cyclic sum on the two sides we get 
$$
\begin{array}{ll}
\oint (\Pi(\mathcal{L}^\rho_{\sharp_\Pi(\alpha)}\beta,\gamma) +\Pi(\beta,\mathcal{L}^\rho_{\sharp_\Pi(\alpha)}\gamma))= & 2\oint \left(\rho_\Pi(\alpha)\cdot \Pi(\beta,\gamma)-\Pi([\alpha,\beta]_\Pi,\gamma)\right) \\ 
 & +3[\Pi,\Pi](\alpha,\beta,\gamma),
\end{array}
$$
and consequently, by (\ref{CrochetPiPi}),
$$
\oint (\Pi(\mathcal{L}^\rho_{\sharp_\Pi(\alpha)}\beta,\gamma) +\Pi(\beta,\mathcal{L}^\rho_{\sharp_\Pi(\alpha)}\gamma))=[\Pi,\Pi](\alpha,\beta,\gamma).
$$
Substituting in the expression of $J_\Pi$ above we get  
$$
\begin{array}{ll}
J_\Pi(\alpha,\beta,\gamma)= & \oint \mathcal{L}^\rho_{\left(\frac{1}{2}[\Pi,\Pi]-\xi\wedge\Pi\right)(\alpha,\beta,.)}\gamma+\oint\gamma(\xi)[\alpha,\beta]_\Pi + \oint \Pi(\alpha,\beta)\mathcal{L}^\rho_\xi\gamma \\ 
 &  + d_\rho\left((2\xi\wedge\Pi-[\Pi,\Pi])(\alpha,\beta,\gamma)\right). 
\end{array} 
$$
In remains only to substitute this in the the expression of $J^\lambda_{\Pi,\xi}$ above to get the result.   
\end{proof}

\begin{cor}
Assume that $(\Pi,\xi)$ is a Jacobi structure on a Lie algebroid $(A,\rho,[.,.])$. If $\lambda$ satisfies the property :
$$
\mathcal{L}^\rho_\xi\alpha=[\lambda,\alpha]^\lambda_{\Pi,\xi}\quad \textrm{ for any $A$-1-form $\alpha$} ,
$$
then $(A^\star,\rho_{\Pi,\xi},[.,.]^\lambda_{\Pi,\xi})$ is a Lie algebroid.  
\end{cor}

\begin{proof}
We need only to prove that $J^\lambda_{\Pi,\xi}=0$. From the theorem above we see that   
$$
J_{\Pi,\xi}^\lambda(\alpha,\beta,\gamma)=-\oint\gamma(\xi)\mathcal{L}^\rho_\xi([.,.]_\Pi)(\alpha,\beta).
$$
A direct calculation gives 
$$
\mathcal{L}^\rho_\xi([.,.]_\Pi)(\alpha,\beta)=\mathcal{L}^\rho_{[\xi,\sharp_\Pi(\alpha)]-\sharp_\Pi(\mathcal{L}^\rho_\xi\alpha)}\beta -\mathcal{L}^\rho_{[\xi,\sharp_\Pi(\beta)]-\sharp_\Pi(\mathcal{L}^\rho_\xi\beta) }\alpha  -d_\rho(\mathcal{L}^\rho_\xi \Pi(\alpha,\beta)). 
$$
Since $\mathcal{L}^\rho_\xi \Pi=0$, then $\mathcal{L}^\rho_\xi([.,.]_\Pi)(\alpha,\beta)=0$.  
\end{proof}

\subsection{Contact Lie algebroids}
In this paragraph we use the global language of tensors to
generalize to the Lie algebroid framework the results in
\cite{lichnerowicz} about the contravariant characterization of
almost cosymplectic manifolds and contact manifolds.

\subsubsection{Almost cosymplectic structures on a skew algebroid}
Let $(A,\rho,[.,.])$ be a skew algebroid with the underlying vector
bundle $A$ of odd rank $2m+1$. By an almost cosymplectic structure
on $A$ we mean a pair $(\Omega,\eta)$ of a $1$-form $\eta$ and a
$2$-form $\Omega$ on $A$ such that the $(2m+1)$-form $\eta \wedge
\Omega^m$ is everywhere nonzero. It is then clear that $\Omega$ is
of rank $2m$, and, as in the classical case, that there is a unique
vector field $\xi$ on $A$ such that
\begin{equation}\label{fundamental-section}
i_\xi\Omega=0  \qquad\textrm{ and }\qquad i_\xi\eta=1.
\end{equation}
We call the $A$-vector field $\xi$ the Reeb section or the
fundamental $A$-vector field. With the almost cosymplectic structure
$(\Omega,\eta)$ we associate a pair $(\Pi,\xi)$ where $\xi$ is the
Reeb section and $\Pi$ is the $A$-bivector field, called the
fundamental $A$-bivector field, defined by
$$
\Pi(\alpha,\beta)=\Omega(\sharp_{\Omega,\eta}(\alpha),\sharp_{\Omega,\eta}(\beta)),
$$
where $\sharp_{\Omega,\eta}$ is the inverse isomorphism of the
vector bundle isomorphism $\flat_{\Omega,\eta} : A\rightarrow
A^\star$ defined by
$$
\flat_{\Omega,\eta}(a)=-i_a\Omega +\eta(a)\eta.
$$

\begin{lem}\label{sharp-Pixi-Omegaeta}
We have $\sharp_{\Pi,\xi}=\sharp_{\Omega,\eta}$.
\end{lem}
\begin{proof}
Let $\alpha,\beta \in \Gamma(A^\star)$ and let $a,b$ such that
$\alpha=\flat_{\Omega,\eta}(a)$ and $\beta=\flat_{\Omega,\eta}(b)$.
Notice that $\alpha(\xi)=\flat_{\Omega,\eta}(a)(\xi)=i_\xi\Omega(a)
+\eta(a)\eta(\xi)=\eta(a)$ and likewise $\beta(\xi)=\eta(b)$.
Therefore, we have
$$
\begin{array}{ll}
\beta(\sharp_{\Pi,\xi}(\alpha)) & =
\Pi(\alpha,\beta)+\alpha(\xi)\beta(\xi) \\
& =\Omega(a,b)+\eta(a)\eta(b) \\
& =\flat_{\Omega,\eta}(b)(a) \\
& =\beta\left(\sharp_{\Omega,\eta}(\alpha)\right).
\end{array}
$$
Thus $\sharp_{\Pi,\xi}=\sharp_{\Omega,\eta}$.
\end{proof}

\begin{prop}\label{cosymplectic-fundamental}
Let $(\Omega,\eta)$ be an almost cosymplectic structure on $A$ and
let $(\Pi,\xi)$ be the associated fundamental pair. Let $a,b,c$ be
$A$-vector fields and let $\alpha=\flat_{\Omega,\eta}(a)$,
$\beta=\flat_{\Omega,\eta}(b)$ and $\gamma=\flat_{\Omega,\eta}(c)$.
We have
\begin{enumerate}
\item $\left(\frac{1}{2}[\Pi,\Pi]-\xi \wedge \Pi\right)(\alpha,\beta,\gamma)=\left(d_\rho\Omega+\eta
\wedge \left(d_\rho\eta-\Omega-\mathcal{L}^\rho_\xi
\Omega\right)\right)(a,b,c)$.
\item $\mathcal{L}^\rho_\xi\Pi(\alpha,\beta)=\left(\eta\wedge\mathcal{L}^\rho_\xi\eta-\mathcal{L}^\rho_\xi \Omega \right)(a,b)$.
\end{enumerate}
\end{prop}

\begin{proof}
Notice that since we have $\alpha(\xi)=\eta(a)$,
$\beta(\xi)=\eta(b)$, $\gamma(\xi)=\eta(c)$ and, by the lemma above,
$\sharp_{\Omega,\eta}=\sharp_{\Pi,\xi}$, it comes that
$\sharp_{\Pi}(\alpha)=a-\eta(a)\xi$,
$\sharp_{\Pi}(\beta)=b-\eta(b)\xi$ and
$\sharp_{\Pi}(\gamma)=c-\eta(c)\xi$. Let us now do some
computations. First, we have
$$
\begin{array}{lll}
[\sharp_\Pi(\alpha),\sharp_\Pi(\beta)] &\!\!\!\!\!\!
=\!\!\!&\!\!\!\!\!\!
[a,b]-\eta(a)[\xi,b]+\eta(b)[\xi,a] \\
&&\!\!\!\!\!\!
\left(\rho(b)(\eta(a))-\rho(a)(\eta(b))+\eta(a)\rho(\xi)(\eta(b))-\eta(b)\rho(\xi)(\eta(a))\right)\xi.
\end{array}
$$
Applying $\eta$  and since $\eta(\xi)=1$, we get
$$
\begin{array}{lll}
\eta\left([\sharp_\Pi(\alpha),\sharp_\Pi(\beta)]\right)&=&\eta\left([a,b]\right)-\rho(a)\left(\eta(b)\right)+\rho(b)\left(\eta(a)\right)
\\
&&\eta(a)\mathcal{L}^\rho_\xi\eta(b)-\eta(b)\mathcal{L}^\rho_\xi\eta(a),
\end{array}
$$
and therefore
\begin{equation}\label{etabracket}
\eta\left([\sharp_\Pi(\alpha),\sharp_\Pi(\beta)]\right)=\left(\eta
\wedge \mathcal{L}^\rho_\xi\eta-d_\rho\eta\right)(a,b).
\end{equation}
Also, using $i_\xi\Omega=0$, we get
$$
\begin{array}{lll}
\Omega\left(\left[\sharp_\Pi(\alpha),\sharp_{\Pi}(\beta)\right],\sharp_\Pi(\gamma)\right)
& = & -i_{\sharp_\Pi(\gamma)}\Omega\left(\left[\sharp_\Pi(\alpha),\sharp_\Pi(\beta)\right]\right) \\
& = & -i_{c}\Omega\left(\left[\sharp_\Pi(\alpha),\sharp_\Pi(\beta)\right]\right) \\
& = & \left(\gamma-\eta(c)\eta\right) \left(\left[\sharp_\Pi(\alpha),\sharp_\Pi(\beta)\right]\right),
\end{array}
$$
and then, using (\ref{torsion-alg-alt-asso-Pi}) and the relation (\ref{etabracket}) above, we get
$$
\begin{array}{lll}
\Omega\left(\left[\sharp_\Pi(\alpha),\sharp_{\Pi}(\beta)\right],\sharp_\Pi(\gamma)\right)
& = & -\frac{1}{2}\left[\Pi,\Pi\right](\alpha,\beta,\gamma)+\Pi\left([\alpha,\beta]_\Pi,\gamma\right)\\
&   & -\eta(c)\left(\eta\wedge \mathcal{L}^\rho_\xi\eta-d_\rho\eta\right)(a,b).
\end{array}
$$
On the other hand we clearly have
$$
\rho\left(\sharp_{\Pi}(\alpha)\right)\cdot\Omega\left(\sharp_\Pi(\beta),\sharp_\Pi(\gamma)\right)
=\rho_\Pi(\alpha)\cdot\Pi(\beta,\gamma).
$$
Using the two relations above and the identity (\ref{CrochetPiPi}), we get
$$
d_\rho\Omega\left(\sharp_{\Pi}(\alpha),\sharp_{\Pi}(\beta),\sharp_{\Pi}(\gamma)\right)
=\frac{1}{2}\left[\Pi,\Pi\right](\alpha,\beta,\gamma)-\eta\wedge d_\rho\eta(a,b,c).
$$
Also, observe that since $i_\xi\Omega=0$, using the Cartan formula we have
$$
d_\rho\Omega(a,b,c)=d_\rho\Omega\left(\sharp_{\Pi}(\alpha),\sharp_{\Pi}(\beta),\sharp_{\Pi}(\gamma)\right)+\eta\wedge\mathcal{L}^\rho_\xi\Omega(a,b,c).
$$
Finally, from the two identities above it follows that
$$
\frac{1}{2}\left[\Pi,\Pi\right](\alpha,\beta,\gamma)
=\left(d_\rho\Omega+\eta\wedge\left(d_\rho\eta-\mathcal{L}^\rho_\xi\Omega\right)\right)(a,b,c),
$$
and it remains only to notice that $\xi\wedge
\Pi(\alpha,\beta,\gamma)=\eta\wedge\Omega(a,b,c)$ to get the first
assertion of the proposition. Let us now prove the second one. On
one hand, we have
\begin{equation}\label{LieXiPi}
\mathcal{L}^\rho_\xi\Pi(\alpha,\beta)=\rho(\xi)\cdot\Pi(\alpha,\beta)
-\Pi\left(\mathcal{L}^\rho_\xi\alpha,\beta\right)-\Pi\left(\alpha,\mathcal{L}^\rho_\xi\beta\right).
\end{equation}
On the other hand, we have
$$
\Pi\left(\mathcal{L}^\rho_\xi\alpha,\beta\right)=-\mathcal{L}^\rho_\xi\alpha\left(\sharp_\Pi(\beta)\right)
=\rho(\xi)\cdot\Pi(\alpha,\beta)+\alpha\left([\xi,\sharp_\Pi(\beta)]\right),
$$
and, since we have $\Pi(\alpha,\beta)=\Omega(a,b)$,
$\sharp_\Pi(\beta)=b-\eta(b)\xi$ and $\alpha(\xi)=\eta(a)$, it
follows that
$$
\begin{array}{lll}
\Pi\left(\mathcal{L}^\rho_\xi\alpha,\beta\right) & = &
\rho(\xi)\cdot\Omega(a,b)+\alpha\left([\xi,b]\right)-\eta(a)\rho(\xi)\left(\eta(b)\right)\\
& = &
\rho(\xi)\cdot\Omega(a,b)+\alpha\left([\xi,b]\right)-\eta(a)\eta([\xi,b])-\eta(a)\mathcal{L}^\rho_\xi\eta(b).
\end{array}
$$
Since $\alpha=\flat_{\Omega,\eta}(a)$, then
$\alpha\left([\xi,b]\right)-\eta(a)\eta([\xi,b])=-i_a\Omega([\xi,b])$,
and therefore
$$
\Pi\left(\mathcal{L}^\rho_\xi\alpha,\beta\right)=\rho(\xi)\cdot\Omega(a,b)-\Omega\left(a,[\xi,b]\right)-\eta(a)\mathcal{L}^\rho_\xi\eta(b).
$$
Interchanging $\alpha$ and $\beta$ we also get
$$
\Pi\left(\alpha,\mathcal{L}^\rho_\xi\beta\right)=-\Pi(\mathcal{L}^\rho_\xi\beta,\alpha)
=\rho(\xi)\cdot\Omega(a,b)-\Omega\left([\xi,a],b\right)+\eta(b)\mathcal{L}^\rho_\xi\eta(a).
$$
Substituting in (\ref{LieXiPi}), we get
$$
\mathcal{L}^\rho_\xi\Pi(\alpha,\beta)=-\rho(\xi)\cdot\Omega(a,b)+\Omega\left(a,[\xi,b]\right)
+\Omega\left([\xi,a],b\right)+\eta(a)\mathcal{L}^\rho_\xi\eta(b)-\eta(b)\mathcal{L}^\rho_\xi\eta(a)
$$
which proves the second assertion of the proposition.
\end{proof}

\subsubsection{Contact Lie algebroids}

Let $\left(A,\rho ,\left[ .,.\right] \right) $ be a skew algebroid with
the underlying vector bundle $A$
of odd rank $2m+1$. A contact structure on $A$ is an $A$-form $\eta$ of degree $1$ such that the $A$-form $\eta \wedge \left(
d_\rho\eta \right)^m$ of degree $2m+1$ does not vanish. So, a
contact structure on $A$ is an almost cosymplectic structure $(\Omega,\eta)$ on $A$ 
with $\Omega=d_\rho\eta$. If $(A,\rho,[.,.])$ is a skew (resp. almost Lie, Lie) algebroid and $\eta$ is a contact structure on $A$ we
call $(A,\rho,[.,.] ,\eta)$ a contact skew (resp. almost Lie, Lie) 
algebroid. 

Let $(\Omega,\eta)$ be an almost cosymplectic structure on a skew algebroid $(A,\rho,[.,.])$ and
let $\xi$ be the associated Reeb section. By the relations
(\ref{fundamental-section}) and the Cartan formula (\ref{cartan}) we
have :
\begin{equation}\label{cosymplectic-cartan}
\mathcal{L}^\rho_\xi\Omega =i_\xi\left(d_\rho\Omega\right)
\qquad\textrm{ and }\qquad
\mathcal{L}^\rho_\xi\eta=i_\xi\left(d_\rho\eta\right).
\end{equation}
If $(A,\rho,[.,.])$ is a Lie algebroid and $\eta$ is a contact structure on $A$, i.e. an almost
cosymplectic structure $(\Omega,\eta)$ on $A$ such that
$\Omega=d_\rho\eta$, since we have $d_\rho\left(d_\rho\eta\right)=0$
and $i_\xi\Omega=0$, it follows that
\begin{equation}\label{contact-cartan}
\mathcal{L}^\rho_\xi\Omega=\mathcal{L}^\rho_\xi \left( d_\rho \eta
\right) =0\qquad\textrm{ and }\qquad \mathcal{L}^\rho_\xi \eta =0.
\end{equation}

The following theorem says that on a Lie algebroid the contact structures are precisely the Jacobi structures $(\Pi,\xi)$ such that $\xi\wedge \Pi^m$ is everywhere nonzero.

\begin{thm}
Let $\left(A,\rho,\left[ .,.\right] \right)$ be a Lie algebroid. Let
$(\Omega,\eta)$ be an almost cosymplectic structure on $A$ and
$(\Pi,\xi)$ be the associated fundamental pair. The pair $(\Omega,\eta)$ is a contact structure if and only if the pair $(\Pi,\xi)$ is a Jacobi structure.
\end{thm}
\begin{proof}
Assume that $\Omega=d_\rho\eta$. That $(\Pi,\xi)$ is a Jacobi
structure is a direct consequence of Proposition
\ref{cosymplectic-fundamental} and the relations
(\ref{contact-cartan}) above. Inversely, assume that $(\Pi,\xi)$ is
a Jacobi structure on $A$. Since $\mathcal{L}^\rho_\xi\Pi=0$, by the
second assertion of Proposition \ref{cosymplectic-fundamental}, it
follows that
\begin{equation}\label{omega-eta-L-eta}
\mathcal{L}^\rho_\xi\Omega=\eta\wedge \mathcal{L}^\rho_\xi\eta.
\end{equation}
By the second relation in (\ref{cosymplectic-cartan}), for any
$a\in\Gamma(A)$ we have
$$
\mathcal{L}^\rho_\xi\eta(a)=i_\xi\left(d_\rho\eta\right)(a)=\eta\wedge
i_\xi\left(d_\rho\eta\right)(\xi,a)=\eta\wedge\mathcal{L}^\rho_\xi\eta(\xi,a),
$$
and then, by (\ref{omega-eta-L-eta}) and the first relation in
(\ref{cosymplectic-cartan}), it follows that
$$
\mathcal{L}^\rho_\xi\eta(a)=\mathcal{L}^\rho_\xi\Omega(\xi,a)=i_\xi\left(d_\rho\Omega\right)(\xi,a)=0.
$$
Hence, $\mathcal{L}^\rho_\xi\eta=0$, and by (\ref{omega-eta-L-eta})
again, it follows that $\mathcal{L}^\rho_\xi\Omega=0$. Now, since we
also have $[\Pi,\Pi]=2\xi\wedge\Pi$, by the first assertion of
Proposition \ref{cosymplectic-fundamental}, it follows that $
d_\rho\Omega+\eta\wedge\left(d_\rho\eta - \Omega\right)=0$, and
hence, for $a,b\in\Gamma(A)$, that
$$
d_\rho\Omega(\xi,a,b)+\eta\wedge\left(d_\rho\eta -
\Omega\right)(\xi,a,b)=0.
$$
On the other hand, since $d_\rho\Omega(\xi,a,b)=i_\xi
\left(d_\rho\Omega\right)(a,b)=\mathcal{L}^\rho_\xi\Omega(a,b)=0$ and 
$i_\xi\left(d_\rho\eta\right)=\mathcal{L}^\rho_\xi\eta=0$, and by the
relations (\ref{fundamental-section}), we have
$$
d_\rho\Omega(\xi,a,b)+\eta\wedge\left(d_\rho\eta -
\Omega\right)(\xi,a,b)=\left(d_\rho\eta-\Omega\right)(a,b).
$$
Therefore $\left(d_\rho\eta-\Omega\right)(a,b)=0$.
\end{proof}

\begin{rem}\label{dual-contact-Lie-alg}
\begin{enumerate}
\item If $(A,\rho,[.,.])$  in the theorem above is just a skew algebroid, we see from the proof that $(\Pi,\xi)$ being Jacobi still implies $(\Omega,\eta)$ is contact. 
\item With a contact Lie algebroid $(A,\rho,[.,.],\eta)$ we naturally associate
a Lie algebroid structure on the dual bundle $A^\star$ of
$A$ as follows. Let $(\Pi,\xi)$ be the Jacobi structure associated with the contact structure $\eta$. Put $\flat_\eta=\flat_{\Omega,\eta}$ and $\sharp_{\eta}=\sharp_{\Omega,\eta}(=\flat_\eta^{-1})$ with $\Omega=d_\rho\eta$. By
Lemma \ref{sharp-Pixi-Omegaeta},  we have $\sharp_{\Pi,\xi}=\sharp_{\eta}$, so
$\sharp_{\Pi,\xi}$ is an isomorphism satisfying $\sharp_{\Pi,\xi}(\eta)=\xi$,
and hence, by Theorem \ref{alg-alt-jacobiO}, an isomorphim satisfying
$$
\sharp_{\Pi,\xi}\left(\left[\alpha,\beta\right]_{\Pi,\xi }^\eta\right)= \left[\sharp_{\Pi,\xi}(\alpha),\sharp_{\Pi,\xi}(\beta)\right].
$$
Therefore, the skew algebroid $(A^\star,\rho_{\Pi,\xi},[.,.]_{\Pi,\xi}^\eta)$
is a Lie algebroid isomorphic to the Lie algebroid $(A,\rho,[.,.])$. If we put $\rho_\eta=\rho\circ \sharp_\eta$ and $[.,.]_\eta=[.,.]_{\Pi,\xi}^\eta$, the Lie algebroid $(A^\star,\rho_\eta,[.,.]_\eta)$
may be called the dual Lie algebroid of the contact Lie algebroid $(A,\rho,[.,.],\eta)$.
\end{enumerate}
\end{rem}

\subsection{Locally conformally symplectic Lie algebroids}

Let $\left(A,\rho ,\left[ .,.\right] \right)$ be a skew (resp. almost Lie, Lie) algebroid on
$M$. A locally conformally symplectic structure on $(A,\rho,[.,.])$ is a pair
$(\Omega,\theta)$ of a closed $1$-form $\theta$ and a nondegenerate 2-form
$\Omega$ on $A$ such that
\begin{equation*}
d_{\rho}\Omega +\theta \wedge \Omega =0.
\end{equation*}%
We also say that $(A,\rho,[.,.],\Omega,\theta)$ is a
locally conformally symplectic skew (resp. almost Lie, Lie) algebroid. In case $\theta$ is
exact, i. e. $ \theta =d_{\rho}f$ for some smooth function $f\in
C^\infty(M)$, we say that $(A,\rho,[.,.],\Omega,\theta)$
is a conformally symplectic skew (resp. almost Lie, Lie) algebroid, it is equivalent to $(
A,\rho,[.,.] ,e^f\Omega)$ being a symplectic skew (resp. almost Lie, Lie) algebroid.

Let $\Omega$ be a nondegenerate 2-form on $A$ and let $\theta \in
\Gamma \left(A^{\star }\right)$. With the pair $(\Omega,\theta)$ we
associate a contravariant pair $(\Pi,\xi) $ as follows :
\begin{equation*}
\Pi(\alpha,\beta)=\Omega \left(
\sharp_\Omega(\alpha),\sharp_\Omega(\beta) \right) \quad\textrm{ and
}\quad \xi =\sharp_\Omega(\theta),
\end{equation*}
where $\sharp_\Omega$ is the inverse isomorphism of the vector
bundle isomorphism $\flat_\Omega :A \rightarrow A^{\star }$,
$\flat_\Omega(a) =-i_a\Omega $. Clearly $\Pi$ is nondegenerate and
$\sharp_\Pi=\sharp_\Omega$. So we have a one-to-one correspondence between such pairs $(\Omega,\theta)$
and the pairs $(\Pi,\xi)$ with $\Pi$ a nondegenerate $A$-bivector field and
$\xi$ an $A$-vector field. We recover the pair $(\Omega,\theta)$ by the relations
$$
\Omega(a,b)=\Pi(\sharp_\Pi^{-1}(a),\sharp_\Pi^{-1}(a)) \quad\textrm{ and }\quad \theta=\sharp_\Pi^{-1}(\xi).
$$
We have the following :

\begin{prop}\label{Omega-theta-Pi-xi}
Let $(A,\rho,[.,.])$ be a skew algebroid. Let $(\Omega,\theta)$ and $(\Pi,\xi)$ be two pairs as above. Let
$a,b,c$ be $A$-vector fields and let $\alpha=\flat_\Omega(a)$,
$\beta=\flat_\Omega(b)$ and $\gamma=\flat_\Omega(c)$. We have
\begin{enumerate}
\item $\left( \frac{1}{2}\left[ \Pi ,\Pi \right] -\xi \wedge \Pi
\right)( \alpha ,\beta ,\gamma)=\left( d_{\rho}\Omega +\theta \wedge
\Omega \right)(a,b,c)$.
\item $\mathcal{L}^\rho_\xi\Pi(\alpha,\beta)=-\mathcal{L}^\rho_\xi\Omega(a,b)$.
\end{enumerate}
\end{prop}
\begin{proof}
Using the identity (\ref{torsion-alg-alt-asso-Pi}), we get
$$
\Omega([a,b],c)=\gamma([a,b])=-\frac{1}{2}\left[\Pi,\Pi\right](\alpha,\beta,\gamma)+\Pi(\left[\alpha,\beta\right]_\Pi,\gamma).
$$
We also have $\rho(a)\cdot\Omega(b,c)=\rho_\Pi(\alpha)\cdot\Pi(\beta,\gamma)$. Therefore, by a direct calculation using the identity (\ref{CrochetPiPi}), we get
\begin{equation}\label{domega}
d_\rho\Omega(a,b,c)=\frac{1}{2}\left[\Pi,\Pi\right](\alpha,\beta,\gamma).
\end{equation}
On the other hand, notice that
$\theta(a)=-i_\xi\Omega(a)=i_a\Omega(\xi)=-\alpha(
\xi)$, likewise $\theta(b)=-\beta(\xi)$ and
$\theta(c)=-\gamma(\xi)$, thus
$\theta\wedge\Omega(a,b,c)=-\xi\wedge\Pi(\alpha,\beta,\gamma)$.
Hence, with (\ref{domega}), we get the first assertion of the proposition.
For the second assertion, is suffices to notice that
$$
\Pi(\mathcal{L}^\rho_\xi\alpha,\beta)=-\mathcal{L}^\rho_\xi\alpha(b)
=-\rho(\xi)(\alpha(b))+\alpha(\mathcal{L}^\rho_{\xi}b)=\rho(\xi)\cdot\Omega(a,b)-\Omega(a,\mathcal{L}^\rho_\xi
b).
$$
\end{proof}

The following theorem shows that the
locally conformally symplectic structures are precisely the Jacobi
structures $(\Pi,\xi)$ with a nondegenerate underlying bivector field $\Pi$.

\begin{thm}
The pair $(\Omega,\theta)$ is a locally conformally symplectic
structure if and only if the pair $(\Pi,\xi)$ is Jacobi.
\end{thm}
\begin{proof}
From the first assertion of Proposition \ref{Omega-theta-Pi-xi} we
deduce that the identity $d_\rho\Omega+\theta \wedge \Omega =0$ is
satisfied if and only if  the identity $\left[ \Pi ,\Pi \right]
=2\xi \wedge \Pi$ is, and if one of the two is satisfied then, using
the Cartan formula, we get
$$
\mathcal{L}^\rho_\xi\Omega  =d_\rho(i_\xi \Omega)+i_\xi d_\rho\Omega =-d_\rho\theta -
i_\xi(\theta\wedge\Omega) =-d_\rho\theta,
$$
and then, with the assertion 2. of Proposition \ref{Omega-theta-Pi-xi},
that $\mathcal{L}^\rho_\xi\Pi=0$ if and only if $d_\rho\theta=0$.
\end{proof}

\begin{rem}\label{dual-locally-conf-sympl-Lie-alg}
Let $\left(A,\rho,[.,.],\Omega,\theta\right) $ be a locally
conformally symplectic skew (resp. almost Lie, Lie) algebroid and let $(\Pi,\xi)$ be the
associated Jacobi structure. Since we have
$\sharp_{\Pi,\xi}(\theta)=\sharp_\Pi(\theta)+\theta(\xi)\xi=\sharp_\Pi(\theta)=\xi$,
by Theorem \ref{alg-alt-jacobiO}, we have
$$
\sharp_{\Pi,\xi}\left(\left[\alpha,\beta\right]_{\Pi,\xi }^\theta\right)= \left[\sharp_{\Pi,\xi}(\alpha),\sharp_{\Pi,\xi}(\beta)\right].
$$
On the other hand, the vector bundle morphism $\sharp_{\Pi,\xi}$ is an ismorphism. Indeed, since we have $\sharp_{\Pi,\xi}(\alpha)=\sharp_\Pi(\alpha+\alpha(\xi)\xi)$ and since the bivector field $\Pi$ is nondegenerate, then $\sharp_{\Pi,\xi}(\alpha)=0$ implies $\alpha=-\alpha(\xi)\theta$, thus $\alpha(\xi)=-\alpha(\xi)\theta(\xi)=0$, and therefore $\alpha=0$. Hence, just as in the contact situation, Remark \ref{dual-contact-Lie-alg}(2), the skew algebroid  $\left(A^\star,\rho_{\Pi,\xi} ,\left[ .,.\right]_{\Pi,\xi}^\theta\right)$ is a skew (resp. almost Lie, Lie)  algebroid isomorphic to the skew (resp. almost Lie, Lie) algebroid $\left(A,\rho ,\left[ .,.\right]\right)$. If we put $\rho_{\Omega,\theta}=\rho_{\Pi,\xi}$ and $[.,.]_{\Omega,\theta}=[.,.]_{\Pi,\xi}^\theta$, the skew (resp. almost Lie, Lie) algebroid $(A^\star,\rho_{\Omega,\theta},[.,.]_{\Omega,\theta}^\theta)$ may be called the dual skew (resp. almost Lie, Lie) algebroid of the locally conformally symplectic skew (resp. almost Lie, Lie) algebroid $(A,\rho,[.,.],\Omega,\theta)$.
\end{rem}

\section{Compatibility of a Jacobi structure and a Riemannian metric}

\subsection{Compatibility of the triple $(\Pi,\xi,g)$}

Let $g$ be a pseudo-Riemannian metric on a skew algebroid
$\left(A,\rho,[.,.]\right)$. Let $\Pi$ be an $A$-bivector field and
$\xi$ be an $A$-vector field. With the triple $(\Pi,\xi,g)$ we
associate the $A$-$1$-form $\lambda$ defined by
\begin{equation*}
\lambda :=g(\xi,\xi) \flat_g(\xi)-\flat_g(J\xi),
\end{equation*}
and we use the notation $[.,.]_{\Pi,\xi}^g$ instead of $[.,.]_{\Pi
,\xi }^\lambda$. The Levi-Civita contravariant $A$-connection
associated with the triple $(\Pi,\xi,g)$ is the
Levi-Civita connection $\mathcal{D}$ of the pseudo-Riemannian skew
algebroid $(A^\star,\rho_{\Pi,\xi},[.,.]_{\Pi,\xi}^g)$. When
$\xi=0$, it is just the Levi-Civita contravariant connection
associated with the pair $(\Pi,g)$ in \S\,\ref{contravariant-connection-Pi-g}

\begin{prop}\label{Levi-Civita-triplet}
Assume that the vector bundle morphism $\sharp_{\Pi,\xi}$ is a skew
algebroid morphism, i.e. we have
$$
\sharp_{\Pi,\xi}\left([\alpha,\beta]_{\Pi,\xi}^{g}\right)=\left[\sharp_{\Pi,\xi}(\alpha)
,\sharp_{\Pi,\xi}(\beta) \right].
$$
Assume also that $\sharp_{\Pi,\xi}$ is an isometry. Then
\begin{equation*}
\sharp_{\Pi,\xi}\left( \mathcal{D}_\alpha \beta \right)=\nabla_{\sharp_{\Pi,\xi}(\alpha)}\left(\sharp_{\Pi,\xi}(\beta)\right) ,
\end{equation*}
where $\nabla$ is the Levi-Civita connection of $(A,\rho,[.,.],g) $.
\end{prop}

\begin{proof}
Since $\sharp_{\Pi,\xi}$ is an isometry we have $g(\sharp_{\Pi,\xi}(\mathcal{D}_\alpha \beta),\sharp_{\Pi,\xi}(\gamma))=g^\star(\mathcal{D}_\alpha \beta,\gamma)$. Now, using the Koszul formula (\ref{koszul}) for the Levi-Civita connections $\mathcal{D}$ and $\nabla$, it follows that
$$
g\left(\sharp_{\Pi,\xi}\left(\mathcal{D}_{\alpha}\beta \right)
,\sharp_{\Pi,\xi}(\gamma) \right) =g\left( \nabla_{\sharp_{\Pi,\xi}(\alpha)}\left(\sharp_{\Pi,\xi}(\beta)\right),\sharp_{\Pi,\xi}(\gamma)\right),
$$
for any $\alpha ,\beta ,\gamma\in\Gamma(A^\star)$.
\end{proof}

We say that the metric $g$ is compatible with the pair
$\left(\Pi,\xi\right) $ or that the triple $\left(\Pi,\xi,g\right)$
is compatible if
\begin{equation}\label{compatibilte du triplet}
\mathcal{D}\Pi(\alpha,\beta,\gamma) =\frac{1}{2}\left(\gamma(\xi)
\Pi(\alpha,\beta) -\beta(\xi)\Pi(\alpha,\gamma)-J^{\star}\gamma(\xi)
g^{\star}(\alpha,\beta)+J^{\star}\beta(\xi)
g^{\star}(\alpha,\gamma)\right),
\end{equation}
for every $\alpha ,\beta ,\gamma \in \Gamma(A^{\star})$. The formula
(\ref{compatibilte du triplet}) can also be written in the form
\begin{equation} \label{compatibilite du triplet 2}
\left(\mathcal{D}_{\alpha }J^\star\right) \beta
=\frac{1}{2}\left(\Pi(\alpha ,\beta) \flat_g(\xi) -\beta(\xi)
J^{\star}\alpha -g^{\star}(\alpha,\beta) J^{\star}\flat_g(\xi)
+J^{\star}\beta(\xi) \alpha\right),
\end{equation}
for every $\alpha ,\beta \in \Gamma(A^\star)$.

The compatibility in the case $\xi$ is the zero $A$-vector field means
that  $\left(A,\rho,[.,.],\Pi,g\right)$ is a pseudo-Riemannian
Poisson skew algebroid, and Riemannian Poisson if moreover the metric
$g$ is positive definite.

The following result generalizes the proposition \ref{bivector-compatibility-Poisson}.

\begin{thm}
Let $\left(A,\rho,[ .,.] ,g\right)$ be a pseudo-Riemannian skew
algebroid, $\Pi$ an $A$-bivector field and $\xi$ an $A$-vector field. Assume that $\mathcal{L}^\rho_\xi\Pi=0$  and $(\Pi,\xi,g)$ is compatible. Then, $(\Pi,\xi)$ is a Jacobi structure if and only if $(\xi-\sharp_\Pi(\lambda))\wedge \Pi=0$.    
\end{thm}

\begin{proof} 
Taking the cyclic sum on the two sides of the equality (\ref{compatibilte du triplet}) we get 
\begin{equation}\label{cyclic-compatibility}
\oint \mathcal{D}\Pi(\alpha,\beta,\gamma)=\xi \wedge \Pi(\alpha,\beta,\gamma). 
\end{equation}
On the other hand, we have  
$$
\begin{array}{ll}
\mathcal{D}\Pi(\alpha,\beta,\gamma) & =\rho_{\Pi,\xi}(\alpha)\cdot \Pi(\beta,\gamma)-\Pi(\mathcal{D}_\alpha \beta,\gamma)-\Pi(\beta,\mathcal{D}_\alpha \gamma) \\ 
 & =\rho_{\Pi}(\alpha)\cdot \Pi(\beta,\gamma)-\Pi(\mathcal{D}_\alpha \beta,\gamma)+\Pi(\mathcal{D}_\alpha \gamma,\beta)+\alpha(\xi)(\rho(\xi)\cdot\Pi(\beta,\gamma).  
\end{array} 
$$
Hence, the connection $\mathcal{D}$ being torsion free, 
$$
\oint \mathcal{D}\Pi(\alpha,\beta,\gamma)=\oint\left(\rho_{\Pi}(\alpha)\cdot \Pi(\beta,\gamma)-\Pi([\alpha,\beta]^\lambda_{\Pi,\xi},\gamma)+\alpha(\xi)(\rho(\xi)\cdot\Pi(\beta,\gamma)) \right). 
$$
Since  
$$
\begin{array}{ll}
\Pi([\alpha,\beta]^\lambda_{\Pi,\xi},\gamma) = & \Pi([\alpha,\beta]_\Pi,\gamma)+\alpha(\xi)\Pi(\mathcal{L}^\rho_\xi\beta,\gamma)+\beta(\xi)\Pi(\gamma,\mathcal{L}^\rho_\xi \alpha) \\  
& - \alpha(\xi)\Pi(\beta,\gamma)-\beta(\xi)\Pi(\gamma,\alpha)-\Pi(\alpha,\beta)\Pi(\lambda,\gamma), 
\end{array}
$$
it follows using (\ref{CrochetPiPi}) that
$$
\oint \mathcal{D}\Pi(\alpha,\beta,\gamma)=\left(2\xi \wedge \Pi - [\Pi,\Pi]\right)(\alpha,\beta,\gamma)+\left(\sharp_\Pi(\lambda)\wedge\Pi\right)(\alpha,\beta,\gamma)+\oint \alpha(\xi)\mathcal{L}^\rho_\xi\Pi(\beta,\gamma), 
$$ 
and, by (\ref{cyclic-compatibility}) and since $\mathcal{L}^\rho_\xi\Pi=0$, that $2\xi\wedge\Pi-[\Pi,\Pi]=\left(\xi-\sharp_\Pi(\lambda)\right)\wedge \Pi$. 
\end{proof}

\subsection{$(1/2)$-Kenmostsu Lie algebroids}

For the notions of almost contact pseudo-Riemannian structures,
contact pseudo-Riemannian structures and Kenmotsu structures on a
smooth manifold, see \cite{blair}. For the definition of these same
structures on a Lie algebroid, see \cite{ida-popescu2}.

\subsubsection{Almost contact and contact Riemannian Lie algebroids}
Let $(A,\rho,[.,.])$ be a skew algebroid. Let $\left(\Phi ,\xi ,\eta
\right)$ be a triple consisting of an $A$-1-form $\eta$, an
$A$-vector field $\xi$ and an $A$-(1,1)-tensor field $\Phi$. The
triple $\left(\Phi ,\xi ,\eta \right)$ is an almost contact
structure on $A$ if
$\Phi^2=-\textrm{Id}_A+\eta\otimes\xi$ and $\eta(\xi)=1$. From what
it follows, in a similar manner as in the case of almost contact
manifolds \cite[Th.4.1]{blair}, that $\Phi(\xi) =0$ and $\eta \circ
\Phi =0$.

We say that a metric $g$ is associated with the triple
$(\Phi,\xi,\eta)$ if the following identity is verified
\begin{equation}\label{met-ass-presque-cont}
g(\Phi(a),\Phi(b))=g(a,b)-\eta(a)\eta(b) .
\end{equation}
Let $(A,\rho,[.,.])$ be a skew (resp. almost Lie, Lie) algebroid. We say that $\left(A,\rho,[ .,.],\Phi ,\xi ,\eta ,g\right) $ is an
almost contact (pseudo-)Riemannian skew (resp. almost Lie, Lie) algebroid if the triple
$\left( \Phi ,\xi ,\eta \right) $ is an almost contact structure on
$A$ and $g$ is an associated (pseudo-)Riemannian metric.

Let $\left(A,\rho ,[ .,.] ,\Phi ,\xi ,\eta ,g\right)$ be an almost
contact pseudo-Riemannian skew algebroid. Notice that if we set
$b=\xi$ in (\ref{met-ass-presque-cont}), we deduce that
\begin{equation*}
g(a,\xi) =\eta(a),
\end{equation*}
for any $a\in \Gamma(A)$, i.e., $\flat_g(\xi)=\eta$, and in
particular that $g(\xi,\xi) =1$. Also, using
(\ref{met-ass-presque-cont}), $\Phi^2=-\textrm{Id}_A+\eta\otimes\xi$
and $\eta \circ \Phi =0$, we get that $g(\Phi(a),b)=-g(a,\Phi(b))$,
for any $a,b\in \Gamma(A)$. So, the map $\Pi:\Gamma(A^\star)\times
\Gamma(A^\star)\rightarrow C^\infty(M)$ defined by
$$
\Pi(\alpha,\beta)=g(\sharp_g(\alpha),\Phi(\sharp_g(\beta)))
$$
is a bivector field on $A$. Let us call it the fundamental bivector
field of the almost contact pseudo-Riemannian skew algebroid
$\left(A,\rho ,[ .,.] ,\Phi ,\xi ,\eta ,g\right)$. We have the
following result

\begin{prop}\label{pi-ass-presque-contact}
The vector bundle morphism $\sharp_{\Pi,\xi}$ is an isometry.
Therefore, if in addition $\sharp_{\Pi,\xi}$ is a skew algebroid
morphism, then
\begin{equation*}
\sharp_{\Pi,\xi}\left( \mathcal{D}_{\alpha }\beta \right)
=\mathcal{\nabla}_{\sharp_{\Pi,\xi}(\alpha)}\left(\sharp_{\Pi,\xi}(\beta)\right),
\end{equation*}
for every $\alpha ,\beta \in \Gamma(A^\star)$.
\end{prop}

\begin{proof}
Let us prove that $\sharp_{\Pi,\xi}$ is an isometry. Let $\alpha\in
\Gamma(A^\star)$. Recall that by definition, we have
$\sharp_{\Pi,\xi}(\alpha)=\sharp_\Pi(\alpha)+\alpha(\xi)\xi$. Since
we have on one hand
$\alpha(\xi)=g(\sharp_g(\alpha),\xi)=\eta(\sharp_g(\alpha))$ and on
the other hand, for any $\beta\in \Gamma(A^\star)$,
$$
\beta(\sharp_\Pi(\alpha))=g(\sharp_g(\alpha),\Phi(\sharp_g(\beta)))
=-g(\Phi(\sharp_g(\alpha)),\sharp_g(\beta))=-\beta(\Phi(\sharp_g(\alpha))),
 $$
i.e. $\sharp_\Pi(\alpha)=-\Phi(\sharp_g(\alpha))$, we deduce that
\begin{equation}\label{f1}
\sharp_{\Pi,\xi}(\alpha)=-\Phi(\sharp_g(\alpha))+\eta(\sharp_g(\alpha))\xi.
\end{equation}
Let $\alpha,\beta\in \Gamma(A^\star)$. From the formula (\ref{f1})
and the fact that $g(\Phi(a),\xi)=\eta\circ\Phi(a)=0$ and
$g(\xi,\xi)=1$, we deduce that
$$
g(\sharp_{\Pi,\xi}(\alpha),\sharp_{\Pi,\xi}(\beta))
=g(\Phi(\sharp_g(\alpha)),\Phi(\sharp_g(\beta)))
+\eta(\sharp_g(\alpha))\eta(\sharp_g(\beta)).
$$
By using the formula (\ref{met-ass-presque-cont}), we get
$$
g(\sharp_{\Pi,\xi}(\alpha),\sharp_{\Pi,\xi}(\beta))
=g(\sharp_g(\alpha),\sharp_g(\beta))=g^\ast(\alpha,\beta).
$$
The second claim  is a direct consequence of Proposition
\ref{Levi-Civita-triplet}.
\end{proof}

Let $(A,\rho,[.,.])$ be a skew (resp. almost Lie, Lie) algebroid. Assume that $\eta$ is a contact form on 
$(A,\rho,[.,.])$. We say that $(A,\rho,[.,.],\eta,g)$ is a contact
pseudo-Riemannian skew (resp. almost Lie, Lie) algebroid, or that the metric $g$ is
associated with the contact form $\eta$, if there exists a vector
bundle endomorphism $\Phi$ of $A$ such that $(\Phi,\xi,\eta,g)$ is an
almost contact pseudo-Riemannian structure on $A$ and that
\begin{equation}\label{met-ass-cont}
g(a,\Phi(b))=d_\rho\eta(a,b).
\end{equation}
If in addition $g$ is positive definite, we say that
$(A,\rho,[.,.],\eta,g)$ is a contact Riemannian skew (resp. almost Lie, Lie) algebroid.

\begin{thm}\label{riemannienne-contact-levi-civita}
Assume that $(A,\rho,[.,.],\eta,g)$ is a contact pseudo-Riemannian
Lie algebroid. We have
\begin{equation*}
\sharp_{\eta}\left( \mathcal{D}_{\alpha }\beta \right)
=\mathcal{\nabla}_{\sharp_{\eta}(\alpha)}\left(\sharp_{\eta}(\beta)\right),
\end{equation*}
for every $\alpha ,\beta \in \Gamma(A^\star)$.
\end{thm}
\begin{proof}
Let $(\Pi,\xi)$ be the Jacobi structure associated
with the contact form $\eta$, then $\sharp_\eta=\sharp_{\Pi,\xi}$. Let $(\Phi,\xi,\eta,g)$ be the almost contact pseudo-Riemannian
structure associated with the contact pseudo-Riemannian structure
$(\eta,g)$ on $A$.  By Remark \ref{dual-contact-Lie-alg}(2)
and the proposition above, we need only to prove that
$\Pi$ is the fundamental bivector field associated with $(\Phi,\xi,\eta,g)$ and that
$\eta=\lambda$. Let $\alpha \in \Gamma(A^\star)$ and put
$a=\sharp_\eta(\alpha)$. By using (\ref{met-ass-cont}), we have
$$
\sharp_g(\alpha)=\sharp_g(\flat_\eta(a))=-\sharp_g(i_a
d\eta)+\eta(a)\xi = \Phi(a)+\eta(a)\xi.
$$
Therefore, applying $\Phi$,
$$
\Phi(\sharp_g(\alpha))=\Phi^2(a)=-a+\eta(a)\xi=-\sharp_\eta(\alpha)+\alpha(\xi)\xi=-\sharp_\Pi(\alpha).
$$
We deduce that
$\Pi(\alpha,\beta)=g(\sharp_g(\alpha),\Phi(\sharp_g(\beta)))$ for
any $\alpha,\beta\in \Gamma(A^\star)$, and hence, that $\Phi=-J$. Therefore, $J\xi=0$, and since $g(\xi,\xi)=1$ it follows that $\lambda=\flat_g(\xi)=\eta$.
\end{proof}

\subsubsection{$(1/2)$-Kenmostsu Lie algebroids}

Let $(A,\rho,[.,.])$ be a skew (resp. almost Lie, Lie) algebroid. An almost contact 
Riemannian structure $(\Phi,\xi,\eta,g)$ on $A$ is said to be $(1/2)$-Kenmotsu if we have
\begin{equation*}
\left( \nabla _{a}\Phi \right)(b) = \displaystyle\frac{1}{2}\left(
g(\Phi(a),b) \xi -\eta(b)\Phi(a)\right),
\end{equation*}
for any $a,b\in \Gamma(A)$. We say also that $(A,\rho,[.,.],\Phi,\xi,\eta,g)$ is a 
$(1/2)$-Kenmotsu skew (resp. almost Lie, Lie) algebroid.

\begin{lem}\label{phi-sharp-pi}
Let $(A,\rho,[.,.])$ be a skew algebroid. Assume that $(\Phi,\xi,\eta,g)$ is an almost contact
pseudo-Riemannian structure on $A$ and let $\Pi$ be the associated fundamental bivector
 field. If the fiber bundle morphism $\sharp_{\Pi,\xi}$ is a skew algebroid morphism, then
$$
\sharp_{\Pi,\xi}\left((\mathcal{D}_{\alpha}
J^\ast)\beta\right)=-\left(\nabla_{\sharp_{\Pi,\xi}(\alpha)}\Phi\right)(\sharp_{\Pi,\xi}(\beta)),
$$
for every $\alpha,\beta\in\Gamma(A^\star)$.
\end{lem}
\begin{proof}
By using the formula (\ref{f1}) and the fact that we have
$\sharp_g\circ J^\ast=J\circ \sharp_g$, we deduce that $
\sharp_{\Pi,\xi}(J^\ast\alpha) = -\Phi(\sharp_g(J^\ast\alpha))+
\eta(\sharp_g(J^\ast\alpha))\xi=-\Phi(\sharp_\Pi(\alpha))=-\Phi(\sharp_{\Pi,\xi}(\alpha)).
$ Therefore
\begin{equation}\label{pi-xi-J-ast-Phi}
\sharp_{\Pi,\xi}\circ J^\ast=-\Phi \circ \sharp_{\Pi,\xi}.
\end{equation}
Hence, with Proposition \ref{pi-ass-presque-contact}, we
have
\begin{eqnarray*}
\sharp_{\Pi,\xi}\left( \left( \mathcal{D}_{\alpha }J^\ast\right)
\beta \right) &=&\sharp_{\Pi,\xi}\left( \mathcal{D}_{\alpha }\left(
J^\ast\beta \right) \right) -\left( \sharp_{\Pi,\xi}\circ
J^\ast\right) \left( \mathcal{D}_{\alpha}\beta \right) , \\
&=&\mathcal{\nabla}_{\sharp_{\Pi,\xi}(\alpha)}\left(\sharp_{\Pi,\xi}\left(J^\ast\beta\right)\right)
+\Phi\left(\sharp_{\Pi,\xi}\left( \mathcal{D}_{\alpha}\beta \right) \right) , \\
&=&-\mathcal{\nabla }_{\sharp_{\Pi,\xi}(\alpha)}\left( \Phi \left(
\sharp_{\Pi,\xi}(\beta)\right) \right) +\Phi \left(
\mathcal{\nabla }_{\sharp_{\Pi,\xi}(\alpha)}\sharp_{\Pi,\xi}(\beta)\right) , \\
&=&-\left( \mathcal{\nabla }_{\sharp_{\Pi,\xi}(\alpha)}\Phi
\right)(\sharp_{\Pi,\xi}(\beta)).
\end{eqnarray*}
\end{proof}

\begin{prop}
Under the same hypotheses of the above lemma, the compatibility of
the triple $(\Pi,\xi,g)$ is equivalent to
$$
(\nabla_a \Phi)(b)=\dfrac{1}{2}\left(g(\Phi(a),b)\xi
-\eta(b)\Phi(a)\right),
$$
for any $a,b\in \Gamma(A)$, and if moreover the metric $g$ is
positive definite, then the triple $(\Pi,\xi,g)$ is compatible if
and only if the almost contact Riemannian skew algebroid
$(A,\rho,[.,.],\Phi,\xi,\eta,g)$ is $(1/2)$-Kenmotsu.
\end{prop}
\begin{proof}
Since we have $J^\ast\flat_g(\xi)=\flat_g(J\xi)=-\flat_g(\Phi
\xi)=0$ and
$$
J^\ast\beta(\xi)=J^\ast\beta(\sharp_g(\eta))=\eta(\sharp_g(J^\ast\beta))=\eta(J
\sharp_g(\beta))=-\eta(\Phi(\sharp_g(\beta)))=0,
$$
then the formula (\ref{compatibilite du triplet 2}) becomes
$$
\left( \mathcal{D}_{\alpha }J^\ast\right) \beta  =\frac{1}{2}\left(
\Pi(\alpha,\beta)\eta- \beta(\xi)J^\ast\alpha \right).
$$
Applying $\sharp_{\Pi,\xi}$ which by Proposition
\ref{pi-ass-presque-contact} is an isometry and hence an
isomorphism, this last formula is equivalent to
$$
\sharp_{\Pi,\xi}\left(\left( \mathcal{D}_{\alpha }J^\ast\right)
\beta \right) =\frac{1}{2}\left(
\Pi(\alpha,\beta)\sharp_{\Pi,\xi}(\eta) -
\beta(\xi)\sharp_{\Pi,\xi}(J^\ast\alpha) \right).
$$
Now, by Formula (\ref{f1}), we have $\sharp_{\Pi,\xi}(\eta)=\xi$,
and if we put $a=\sharp_{\Pi,\xi}(\alpha)$ and
$b=\sharp_{\Pi,\xi}(\beta) $, then we have $\beta(\xi)=\eta(b)$,
also using (\ref{pi-xi-J-ast-Phi}), we have
$\sharp_{\Pi,\xi}(J^\ast\alpha)=-\Phi(a)$ and
$$
\begin{array}{ll}
\Pi(\alpha,\beta)
& =g(\sharp_g(\alpha),\Phi(\sharp_g(\beta))) \\
& =-g(\sharp_g(\alpha),\sharp_g(J^\ast\beta)) \\
& =-g^\ast(\alpha,J^\ast\beta) \\
& =g(a,\Phi(b)).
\end{array}
$$
It remains to use the lemma above.
\end{proof}

\begin{thm}
Assume that $(A,\rho,[.,.],\eta,g)$ is a contact Riemannian Lie algebroid and
let $(\Phi,\xi,\eta,g)$ be the associated almost contact Riemannian
structure. Assume that $(\Pi,\xi)$ is the Jacobi structure
associated with the contact form $\eta$. Then the triple
$(\Pi,\xi,g)$ is compatible if and only if $(A,\rho,[.,.],\Phi,\xi,\eta,g)$
is $(1/2)$-Kenmotsu.
\end{thm}
\begin{proof}
We have proved, see the proof of Theorem \ref{riemannienne-contact-levi-civita}, that $\Pi$ is the bivector field of the proposition
above and that $\lambda=\eta$.
\end{proof}

\subsection{Locally conformally K\"{a}hler Lie algebroids}

For the notions of almost Hermitian, Hermitian, K\"{a}hler Lie algebroids see \cite{ida-popescu1}. 
In this paragraph $(A,\rho,[.,.])$ is a skew algebroid.

\subsubsection{Riemannian metric associated with a locally conformally symplectic structure}

Assume that $\Omega\in \Gamma(\wedge^2 A^\star)$ is a nondegenerate
$2$-form and let $\theta\in\Gamma(A^\star)$. Assume that the pair
$(\Pi,\xi)$ is associated with the pair $(\Omega,\theta)$. We say
that a pseudo-Riemannian metric $g$ is associated with the pair
$(\Omega,\theta)$ if $\sharp_{\Omega,\theta}:=\sharp_{\Pi,\xi}$ is
an isometry, i.e. if
\begin{equation}\label{met-ass-loc-conf-symp}
g\left(\sharp_{\Omega,\theta}(\alpha),\sharp_{\Omega,\theta}(\beta)\right)
=g^\ast(\alpha,\beta),
\end{equation}
for every $\alpha ,\beta \in \Gamma(A^\star)$.

If $\theta =0$, then $\xi=0$ and
$\sharp_{\Omega,\theta}=\sharp_{\Omega}$, and if $J$ and $J^\ast$
are the fields of endomorphisms defined by the formulae (\ref{J et J
star}), then
$$
\begin{array}{ll}
g(\sharp_{\Omega,\theta}(\alpha),\sharp_{\Omega,\theta}(\beta))
&= g(\sharp_{\Omega}(\alpha),\sharp_{\Omega}(\beta)) \\
&= g^\ast(\flat_g(\sharp_{\Omega}(\alpha)),\flat_g(\sharp_{\Omega}(\beta))) \\
&= g^\ast(J^\ast\alpha,J^\ast\beta),
\end{array}
$$
for any $\alpha ,\beta \in \Gamma(A^\star)$. Hence, in the case
$\theta=0$, the relation (\ref{met-ass-loc-conf-symp}) is equivalent
to
\begin{equation*}
g^\ast\left( J^\ast\alpha ,J^\ast\beta \right) =
g^\ast(\alpha,\beta).
\end{equation*}
If moreover $g$ is positive definite, this last identity means that
the pair $(\Omega,g)$ is an almost Hermitian structure on $A$ and
that $J$ is the associated almost complex structure, i.e., we have
$$
g(Ja,Jb)=g(a,b) \qquad\textrm{ and }\qquad \Omega(a,b)=g(Ja,b),
$$
for every $a,b\in\Gamma(A)$.

\begin{thm}\label{loc-conf-symp-met-ass-levi-civita}
Assume that $(\Omega,\theta)$ is a locally conformally symplectic
structure on $A$ and that $g$ is an associated metric. We have
$$
\sharp_{\Omega,\theta}\left(\mathcal{D}_\alpha\beta\right)
=\nabla_{\sharp_{\Omega,\theta}(\alpha)}\left(\sharp_{\Omega,\theta}(\beta)\right),
$$
for every $\alpha,\beta\in \Gamma(A^\star)$.
\end{thm}
\begin{proof}
By Proposition \ref{Levi-Civita-triplet} and Remark
\ref{dual-locally-conf-sympl-Lie-alg}, we need only to prove that
$\lambda=\theta$. On one hand, we have $
\sharp_{\Pi,\xi}(\theta)=\xi$. On the other hand, for any
$\alpha\in\Gamma(A^\star)$, we have
$$
\begin{array}{ll}
g(\sharp_{\Pi,\xi}(\lambda),\sharp_{\Pi,\xi}(\alpha))&
=g(\sharp_g(\lambda),\sharp_g(\alpha))\\
 & =g(\xi,\xi)\alpha(\xi)+g(\xi,J\sharp_g(\alpha)) \\
 & =g(\xi,\xi)\alpha(\xi)+g(\xi,\sharp_\Pi(\alpha)) \\
 & =g(\xi,\sharp_{\Pi,\xi}(\alpha)).
\end{array}
$$
Since $\sharp_{\Pi,\xi}$ is an isometry, hence an isomorphism, then
$\sharp_{\Pi,\xi}(\lambda)=\xi$.
\end{proof}

\begin{cor}\label{levi-civita-omega}
Under the same hypotheses of the theorem above, we have
$$
\mathcal{D}\Pi(\alpha,\beta,\gamma)=\nabla
\Omega(\sharp_{\Omega,\theta}(\alpha),\sharp_{\Omega,\theta}(\beta),\sharp_{\Omega,\theta}(\gamma)).
$$
\end{cor}
\begin{proof}
We have $\Omega \left( \xi ,\sharp_{\Pi}(\alpha) \right) =-
i_{\sharp_{\Pi}(\alpha)}\Omega(\xi) =\alpha(\xi)$ and likewise
$\Omega(\xi,\sharp_\Pi(\beta))=\beta(\xi)$. Consequently
\begin{equation}\label{omega-pi-xi}
\Omega(\sharp_{\Pi,\xi}(\alpha)
,\sharp_{\Pi,\xi}(\beta))=\Pi(\alpha,\beta).
\end{equation}
It suffices now to compute $\nabla
\Omega(\sharp_{\Pi,\xi}(\alpha),\sharp_{\Pi,\xi}(\beta),\sharp_{\Pi,\xi}(\gamma))
$ and use the theorem above.
\end{proof}

\subsubsection{Locally conformally K\"{a}hler Lie algebroids}

Recall, see \cite{ida-popescu1}, that if $\Omega$ is a nondegenerate
$2$-form on $A$ and $g$ an associated Riemannian metric, the almost
Hermitian structure $(\Omega,g)$ is Hermitian if the associated
almost complex structure is integrable, and K\"ahler if moreover
$\Omega$ is closed. Recall also that if $(\Omega,g)$ is almost
Hermitian, then it is K\"ahler if and only if the $2$-form $\Omega$
is parallel for the Levi-Civita connection of $g$.

If $(\Omega,\theta)$ is a locally conformally symplectic structure and
$(\Omega,g)$ a Hermitian structure, we say that the triple
$(\Omega,\theta,g)$ is a locally conformally K\"ahler structure.

We shall prove that if $(\Omega,\theta)$ is a locally conformally
symplectic structure on $A$ and that $(\Pi,\xi)$ is the associated
Jacobi structure, if $g$ is a Riemannian metric associated with
$\Omega$ and with $(\Omega,\theta)$, the compatibility of the triple
$(\Pi,\xi,g)$ induces a locally conformally K\"ahler structure on $A$.

\begin{lem}
Assume that $\Omega\in \Gamma(\wedge^2A)$ is a nondegenerate
$2$-form on $A$ and let $\theta\in\Gamma(A^\star)$. Assume that
$(\Pi,\xi)$ is the pair associated with $(\Omega,\theta)$. If the
pseudo-Riemannian metric $g$ is associated with the $2$-form
$\Omega$ and with the pair $(\Omega,\theta)$, then we have
$$
J\circ \sharp_{\Pi,\xi}=\sharp_{\Pi,\xi}\circ J^\ast.
$$
\end{lem}
\begin{proof}
Since the metric $g$ is assumed to be associated with $\Omega$ and
using  (\ref{omega-pi-xi}), we get
$$
g(J\sharp_{\Pi,\xi}(\alpha),\sharp_{\Pi,\xi}(\beta))=\Omega(\sharp_{\Pi,\xi}(\alpha),\sharp_{\Pi,\xi}(\beta))=\Pi(\alpha,\beta)=g^\ast(J^\ast\alpha,\beta),
$$
and since  $g$ is also assumed to be associated with the pair
$(\Omega,\theta)$, i.e. $\sharp_{\Pi,\xi}$ is an isometry, then
$$
g(J\sharp_{\Pi,\xi}(\alpha),\sharp_{\Pi,\xi}(\beta))=g(\sharp_{\Pi,\xi}(J^\ast\alpha),\sharp_{\Pi,\xi}(\beta)).
$$
Finally, since $\sharp_{\Pi,\xi}$ is an isometry, hence an
isomorphism, the result follows.
\end{proof}

\begin{thm}
Assume that $(A,\rho,[.,.])$ is an almost Lie algebroid. Let $(\Omega,d_\rho f)$ be a conformally symplectic structure on $A$
and let $(\Pi,\xi)$ be the associated Jacobi structure. If $g$ is a
Riemannian metric associated with $\Omega$ and with $(\Omega ,d_\rho f)$,
then the triple $(\Pi,\xi,g)$ is compatible if and only if the
triple $(\Omega,d_\rho f,g)$ is a conformally K\"ahler structure.
\end{thm}
\begin{proof}
We need to prove that the triple $(\Pi,\xi,g)$ is compatible if and
only if the pair  $(e^{f}\Omega,e^{f}g)$ is compatible, i.e., if and
only if the $A$-$2$-form $e^f\Omega$ is parallel with respect to the
Levi-Civita connection $\nabla^f$ associated with the metric
$g^f=e^fg$. As the connections $\nabla$ and $\nabla^f$ are related
by the formula
$$
\nabla_{a}^{f}b=\nabla_{a}b+\frac{1}{2}\left(d_\rho f(a)b+d_\rho f(b)a-g(a,b)\sharp_g(d_\rho f) \right),
$$
we deduce that
$$
\begin{array}{lll}
\nabla^f\Omega(a,b,c) & = & \nabla\Omega(a,b,c)-d_{\rho}f(a)\Omega(b,c) \\ & & -\frac{1}{2}d_{\rho}f(b)\Omega(a,c)+\frac{1}{2}d_{\rho}f(c)\Omega(a,b) \\
& & +\frac{1}{2}\left(g(a,b)\Omega(\sharp_{g}(d_{\rho}f),c)-g(a,c)\Omega(\sharp_{g}(d_{\rho}f),b)\right),
\end{array}
$$
and then that
$$
\nabla^f (e^f\Omega)(a,b,c)=
e^f\left(d_{\rho}f(a)\Omega(b,c)+\nabla^f\Omega(a,b,c)\right) =e^f
\Lambda_f(a,b,c),
$$
where we have set
$$
\begin{array}{lll}
\Lambda_f(a,b,c) &=& \nabla\Omega(a,b,c)-\dfrac{1}{2}\left(d_{\rho}f(b)\Omega(a,c)-d_{\rho}f(c)\Omega(a,b)\right) \\
& & +\frac{1}{2}\left(g(a,b)\Omega(\sharp_g(d_\rho f),c)-g(a,c)\Omega(\sharp_g(d_\rho f),b)\right).
 \end{array}
$$
It follows that $\nabla^f(e^f\Omega)=0$ if and only if
$\Lambda_f=0$, hence that the pair $(e^f\Omega,e^fg)$ is compatible
if and only if
$$
\begin{array}{lll}
\nabla\Omega(a,b,c) & \!=\! & \dfrac{1}{2}\left(d_{\rho}f(b)\Omega(a,c)-d_{\rho}f(c)\Omega(a,b) -g(a,b)\Omega(\sharp_g(d_\rho f),c) \right. \\
& & \left. +g(a,c)\Omega(\sharp_g(d_\rho f),b)\right).
\end{array}
$$
Let us prove now that this last identity is equivalent to the
formula (\ref{compatibilte du triplet}). Let $\alpha ,\beta, \gamma \in
\Gamma(A^\star)$ be such that $a=\sharp_{\Pi,\xi}(\alpha)$,
$b=\sharp_{\Pi,\xi}(\beta)$ and $c=\sharp_{\Pi,\xi}(\gamma)$. By
Corollary \ref{levi-civita-omega}, we have $\nabla
\Omega(a,b,c)=\mathcal{D}\Pi(\alpha,\beta,\gamma)$. On the other
hand, setting $\theta=d_{\rho}f$, we have $
d_{\rho}f(b)=\theta(b)=\theta(\sharp_\Pi(\beta))+\beta(\xi)\theta(\xi)=-\beta(\sharp_\Pi(\theta))=-\beta(\xi)$
and likewise $d_{\rho}f(c)=-\gamma(\xi)$. Also, by (\ref{omega-pi-xi}), we
have $\Omega(a,b)=\Pi(\alpha,\beta)$ and
$\Omega(a,c)=\Pi(\alpha,\gamma)$. Finally, since the metric $g$ is
associated with $\Omega$, it follows that
$$
\begin{array}{ll}
\Omega(\sharp_g(d_\rho f), b) & =-\Omega(b,\sharp_g(\theta)) \\
& =-g(Jb,\sharp_g(\theta)) \\
& =-\theta(Jb) \\
& =\Omega(\xi,Jb) \\
& =\Omega(\sharp_{\Pi,\xi}(\theta),J\sharp_{\Pi,\xi}(\beta)),
\end{array}
$$
and since $g$ is associated with $\Omega$ and with
$(\Omega,\theta)$, by using the lemma above and (\ref{omega-pi-xi}),
we get
$$
\Omega(\sharp_g(d_\rho f),b)=\Omega(\sharp_{\Pi,\xi}(\theta),\sharp_{\Pi,\xi}(J^\ast\beta))=\Pi(\theta,J^\ast\beta)=J^\ast\beta(\sharp_\Pi(\theta))=J^\ast\beta(\xi)
$$
and likewise $\Omega(\sharp_g(d_\rho f), c)=J^\ast\gamma(\xi)$.
\end{proof}



\bigskip\bigskip

\noindent Yacine A\"it Amrane\\
Laboratoire Alg\`ebre et Th\'eorie des
Nombres\\
Facult\'e de Math\'ematiques\\
USTHB, BP 32, El-Alia, 16111 Bab-Ezzouar, Alger, Algeria\\
e-mail : yacinait@gmail.com \bigskip


\noindent Ahmed Zeglaoui\\
Laboratoire de G\'eom\'etrie, Analyse, Contr\^ole et Applications\\
D\'epartement de Math\'ematiques - Facult\'e des Sciences\\
Universit\'e de Sa\"{\i}da Dr Moulay Tahar, BP 138, Ennasr, 20006
Sa\"{\i}da, Algeria\\
e-mail : ahmed.zeglaoui@gmail.com


\begin{thebibliography}{1}




\bibitem{aitamrane-zeglaoui} Y. A\"{\i}t Amrane, A. Zeglaoui, \textit{Compatibility of Riemannian structures and Jacobi structures.} Journal of Geometry and Physics 133 (2018), 71-80.

\bibitem{blair} D. E. Blair, \textit{Riemannian geometry of contact and
symplectic manifolds.} Pogress in mathematics, vol. 203, 2nd Edition, Birkh\"auser, 2010.

\bibitem{boucetta1} M. Boucetta, \textit{Compatibilit\'{e} des structures
pseudo-riemanniennes et des structures de Poisson.} C. R. Acad. Sci.
Paris, Ser. I 333 (2001), 763--768.

\bibitem{boucetta2} M. Boucetta, \textit{Poisson manifolds with  compatible pseudo-metric and pseudo-Riemannian Lie
algebras.} Differential Geom. Appl. 20 (2004), 279--291.

\bibitem{boucetta3} M. Boucetta, \textit{Riemannian geometry of Lie algebroids.} Journal of the Egyptian Mathematical Society, Vol. 19, Iss. 12 (2011), 57--70.

\bibitem{cannas-weinstein} A. Cannas da Silva, A. Weinstein,
\textit{Geometric models for noncommutative algebras.} Berkeley
Mathematics Lecture Notes 10, Amer. Math. Soc., Providence, RI,
1999.

\bibitem{ida-popescu1} C. Ida, P. Popescu, \textit{On almost complex Lie algebroids.} Mediterr. J. Math. 13 (2016), 803--824.

\bibitem{ida-popescu2} C. Ida, P. Popescu, \textit{Contact structures on Lie
algebroids.} Publicationes Mathematicae 91(1-2)(2017).
DOI:10.5486/PMD.2017.7444

\bibitem{kosmann} Y. Kosmann-Schwarzbach, \textit{Poisson manifolds, Lie algebroids, modular classes : a
survey.} SIGMA 4 (2008), 005.

\bibitem{lichnerowicz} A. Lichnerowicz, \textit{Les vari\'et\'es de Jacobi et leurs alg\`ebres de Lie associ\'ees.} J. Math. pures et appl. 51 (1978), 453--488.

\bibitem{marle0} C.-M. Marle, \textit{Calculus on Lie algebroids, Lie groupoids and Poisson manifolds.} Dissertationes Mathematicae, Vol. 457, Institute of Mathematics, Polish Academy of Sciences, 2008.

\bibitem{vaisman} I. Vaisman, \textit{Lectures on the geometry of Poisson
manifolds} Progress in Mathematics, vol. 118, Birkh\"auser, 1994.



\end{thebibliography}
\end{document}